\documentclass{amsart}
\usepackage[english]{babel}
\usepackage[utf8x]{inputenc}
\usepackage[T1]{fontenc}

\usepackage[a4paper,top=1in,bottom=1in,left=1in,right=1in]{geometry}

\usepackage{amsmath}
\usepackage{indentfirst}
\usepackage{mathrsfs}
\usepackage{graphicx}
\usepackage[colorinlistoftodos]{todonotes}
\usepackage[colorlinks=true, allcolors=blue]{hyperref}
\usepackage{amsfonts}
\usepackage{float}

\usetikzlibrary{calc}

\newtheorem{theorem}{Theorem}[section]

\newtheorem{lemma}[theorem]{Lemma}

\theoremstyle{definition}
\newtheorem{defn}{Definition}[section]

\newtheorem{prop}{Proposition}[section]
\newtheorem{exmp}{Example}[section]
\newtheorem{rem}{Remark}[section]

\title{Connectedness of the Moduli Space of Genus 1 Planar Tropical Curves}
\author{Stanley Wang}

\begin{document}

\begin{abstract}
Tropical geometry is a relatively recent field in mathematics created as a simplified model for certain problems in algebraic geometry. We introduce the definition of  abstract and planar tropical curves as well as their properties, including combinatorial type and degree. We also talk about the moduli space, a geometric object that parameterizes all possible types of abstract or planar tropical curves subject to certain conditions. Our research focuses on the moduli spaces of planar tropical curves of genus one, arbitrary degree $d$ and any number of marked, unbounded edges. We prove that these moduli spaces are connected.
\end{abstract}

\maketitle

\mbox{}\vspace*{-2\baselineskip}

\section{Introduction}

Tropical geometry is a developing field in mathematics, becoming more prominent since the late 1990s. This area of mathematics is dubbed ``tropical'' in honor of the Brazilian mathematician Imre Simon and his numerous contributions to this field. 

Tropical geometry is used to model similar, simpler problems in algebraic geometry by turning questions about algebraic varieties into questions about polyhedral complexes. Tropical geometry is related to other areas of math as well, such as real and complex geometry, and can be used to compute Zeuthen's numbers and Welschniger numbers \cite{mikhalkin2006tropical}. Problems in tropical geometry such as checking tropical connectivity for intersecting hypersurfaces can be solved using algorithms and computing, leading to a better understanding of tropical varieties \cite{theobald2006frontiers}. As for practical applications, tropical geometry has many implications in economics. For example, a tropical line appeared in Elizabeth Baldwin and Paul Klemperer's design of auctions used by the Bank of England in the 2007 financial crisis \cite{baldwin2016understanding}. In addition, Yoshinori Shiozawa found that Ricardian trade theory can be interpreted as subtropical convex algebra \cite{shiozawa2015trade}. Finally, tropical geometry can be used as a framework for optimization problems arising in job scheduling, location analysis, transportation networks, decision making and discrete event dynamical systems.

Tropical curves come in two different types: abstract and planar. An abstract tropical curve can be thought of as a metric graph while a planar tropical curve is an abstract tropical curve along with a mapping of the abstract tropical curve to $\mathbb{R}^2$ giving each edge a two dimensional direction vector. We focus on tropical curves of genus $1$, which means that the curve's subgraph has only one cycle, $\Gamma_c$. The moduli space of tropical curves is a geometric object which allows us to parameterize all the different types of curves and has the properties of a topological space. The moduli space of abstract tropical curves is simple and already understood well, so in this project we investigate the properties of the moduli space of planar tropical curves with genus $1$.

The main result in this paper proves the connectedness of the moduli space of genus $1$ planar tropical curves of any degree: \\

\noindent \textbf{Theorem~\ref{theorem5.3}.} \textit{The moduli space $\mathcal{M}_{=1, n}(d)$ of degree $d$ and genus $1$ planar tropical curves with $n$ marked unbounded edges is connected.} \\

We define abstract tropical curves, planar tropical curves, and the properties they satisfy, as well as the notion of a moduli space in Section 2. In Section 3, we discuss the definitions of degeneration and regeneration, maps we use to manipulate the bounded edges in planar tropical curves while preserving connectedness in the moduli space. We also narrow our scope to tropical curves  with no bounded edges outside of the cycle $\Gamma_c$ and no marked unbounded edges. Then, we find an upper bound for the absolute values of the coordinates of the direction vectors of bounded edges in planar tropical curves and prove that the moduli space is connected in the case that the degree of the planar tropical curve is $1$. In Section 4, we delve deeper into degenerations and regenerations, defining specific types of these maps including three edge degeneration, quadrilateral-triangle degeneration, and triangle-quadrilateral regeneration. In Section 5, we state and prove two lemmas which we use to finally prove that the moduli space of all genus $1$ planar tropical curves of degree $d$ is connected. We discuss the connectedness of moduli spaces when certain constraints on our curves (such as the genus) are lifted in Section 6.

\section{Moduli Space of Tropical Curves}
In this section, we recall some definitions from Gathmann, Kerber, and Markwig in \cite{gathmann2008kontsevich} and \cite{kerber2006counting} about tropical curves and their moduli spaces. We will adopt several definitions regarding topology, continuity, and connectedness from Armstrong in \cite{armstrong2013basic} but will not state these definitions for the sake of brevity.

\subsection{Metric Graphs}

\begin{defn}
A metric graph $\Gamma(V, E)$ consists of a set $V$ of points called \textit{vertices} and a set $E$ of lines called \textit{edges} that are connected to one or two vertices. In a metric graph, an edge that connects two vertices is called a \textit{bounded edge} and is associated with a positive real number called its length. A bounded edge of length $l$ can be thought of as an interval $[0, l]$ of the real line. An edge can also be \textit{unbounded}, which means it is connected to only one vertex and is not given a length, corresponding to the interval $[0, \infty)$. A \textit{flag} $F$ consists of a vertex and an edge emanating from that vertex. The \textit{valence} of a vertex $V$ is the number of of flags $F$ with $V$ as its vertex. The \textit{genus} of a graph is defined to be $1 + E - V$ where $E$ is the number of bounded edges and $V$ is the number of vertices. The genus can also be thought of as the number of cycles in the graph. For example, a tree has genus $0$ because it has no cycles and has one less bounded edge than it has vertices. 
\end{defn}

\subsection{Abstract Tropical Curves}

\begin{defn}
We define an \textit{abstract tropical curve} as a connected metric graph with each vertex having valence at least $3$. An abstract tropical curve can be \textit{$n$-marked}, which means that $n$ of the graph's unbounded edges are marked and distinguishable. An $n$-marked abstract curve is denoted by $(\Gamma, x_1, x_2, \ldots , x_n)$ where $\Gamma$ is the graph and the $x_i$ are marked, unbounded edges. Here is an example of the simplest abstract tropical curve.
\end{defn}

\begin{figure}[ht]
\centering
\begin{tikzpicture}

\draw[dashed] (0, 0) -- (0.85, 0.5) node [at end] {$x_1$};

\draw[dashed] (0, 0) -- (-0.85, 0.5) node [at end] {$x_2$};

\draw[dashed] (0, 0) -- (0, -1) node [at end] {$x_3$};

\end{tikzpicture}
\setlength{\belowcaptionskip}{-3pt}
\caption{Abstract tropical curve with one vertex and three marked, unbounded edges.}
\end{figure}
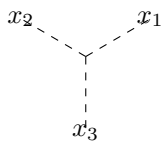

\begin{defn}
The \textit{combinatorial type} of an abstract tropical curve $(\Gamma, x_1, x_2, \ldots , x_n)$ is defined to be the homeomorphism class of $\Gamma$ relative $x_i$ (i.e. the data of $(\Gamma, x_1, x_2, \ldots , x_n)$ that map each $x_i$ to itself). Combinatorial type can also be thought of as the structure of the original graph, disregarding the lengths of the bounded edges.
\end{defn}

\begin{defn}
The \textit{moduli space} of abstract tropical curves with $n$ marked edges and genus $g$ and with all unbounded edges marked is a geometric space whose points represent isomorphism classes of such curves and is denoted by $\mathcal{M}_{g, n}$. The moduli space $\mathcal{M}_{g, n}$ is obtained by gluing together finitely many orthants $\mathbb{R}_{\ge 0}^m$, one orthant for each combinatorial type in the moduli space \cite{Brodsky2015}.
\end{defn}

\begin{exmp}
The moduli space $\mathcal{M}_{0,3}$ consists of a single point, as the only possible abstract tropical curve with genus $0$ and three marked edges is the one shown in Figure 1.
\end{exmp}

\begin{exmp}
The moduli space $\mathcal{M}_{0, 4}$ consists of three rays with a common vertex. Given an abstract tropical curve $(\Gamma, x_1, x_2, x_3, x_4)$ with four marked edges, there can be at most one bounded edge and at most two vertices in order to ensure each vertex has valence at least $3$. The abstract tropical curve must look like one of the following three curves (labeled $A$, $B$, $C$):

\begin{figure}[ht]
\centering
\begin{tikzpicture}

\draw (-0.5, 0) -- (0.5, 0);

\draw[dashed] (-0.5, 0) -- (-1.3, 0.5) node [at end, left] {$x_1$};

\draw[dashed] (-0.5, 0) -- (-1.3, -0.5) node [at end, left] {$x_2$};

\draw[dashed] (0.5, 0) -- (1.3, 0.5) node [at end, right] {$x_3$};

\draw[dashed] (0.5, 0) -- (1.3, -0.5) node [at end, right] {$x_4$};

\node at (0, -0.5) {$A$};

\end{tikzpicture}
\hspace{0.5cm}
\begin{tikzpicture}

\draw (-0.5, 0) -- (0.5, 0);

\draw[dashed] (-0.5, 0) -- (-1.3, 0.5) node [at end, left] {$x_1$};

\draw[dashed] (-0.5, 0) -- (-1.3, -0.5) node [at end, left] {$x_3$};

\draw[dashed] (0.5, 0) -- (1.3, 0.5) node [at end, right] {$x_2$};

\draw[dashed] (0.5, 0) -- (1.3, -0.5) node [at end, right] {$x_4$};

\node at (0, -0.5) {$B$};

\end{tikzpicture}
\hspace{0.5cm}
\begin{tikzpicture}

\draw (-0.5, 0) -- (0.5, 0);

\draw[dashed] (-0.5, 0) -- (-1.3, 0.5) node [at end, left] {$x_1$};

\draw[dashed] (-0.5, 0) -- (-1.3, -0.5) node [at end, left] {$x_4$};

\draw[dashed] (0.5, 0) -- (1.3, 0.5) node [at end, right] {$x_3$};

\draw[dashed] (0.5, 0) -- (1.3, -0.5) node [at end, right] {$x_2$};

\node at (0, -0.5) {$C$};

\end{tikzpicture}
\setlength{\belowcaptionskip}{-3pt}
\caption{Three types of abstract tropical curves.}
\end{figure}
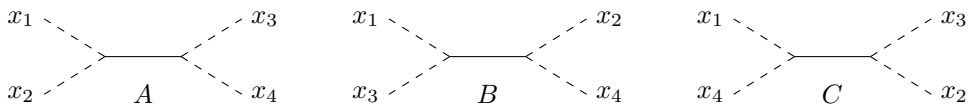

The only parameter that may vary besides the type of abstract tropical curve is the length of the single bounded edge, which can be any positive real number. Each ray in the moduli space corresponds to one of the above types, and the distance from a point on the ray to the vertex of the ray parameterizes the length of the bounded edge. Of course, the vertex common to all three rays represents the abstract tropical curve with one vertex, no bounded edges, and all marked edges attached to the one vertex.

\begin{figure}[ht]
\centering
\begin{tikzpicture}

\draw (0, 0) -- (0.85, 0.5);

\draw (0, 0) -- (-0.85, 0.5);

\draw (0, 0) -- (0, -1);

\draw[fill=black] (0, 0) circle (0.03 cm);

\end{tikzpicture}
\setlength{\belowcaptionskip}{-3pt}
\caption{The moduli space $\mathcal{M}_{0, 4}$.}
\end{figure}
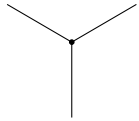
\end{exmp}

\begin{rem}
Caporaso proved the connectedness of the moduli space of abstract tropical curves in \cite{caporaso2011algebraic}. We build on these results by proving connectedness for a different type of moduli space, one that parameterizes planar tropical curves.
\end{rem}

\subsection{Planar Tropical Curves}

\begin{defn}
A \textit{planar tropical curve} consists of an $n$-marked abstract tropical curve \linebreak $(\Gamma, x_1, x_2, \ldots , x_n)$ along with a mapping $h$ taking the abstract tropical curve to $\mathbb{R}^2$. Planar tropical curves are denoted by $(\Gamma, h, x_1, x_2, \ldots , x_n)$, but we may refer to a planar tropical curve simply by its subgraph $\Gamma$. Unbounded edges in planar tropical curves can be marked, which means the edges are distinguishable, or unmarked, which means the edges are indistinguishable. In the plane, each edge is given an integer direction vector, which is a vector of the form $(a, b)$ where $a, b \in \mathbb{Z}$. For a flag $F$ we denote the integer direction vector of $F$ as $v(F)$. Edges of length $l$ are mapped to $a + v \cdot l$ where $a$ is some point in $\mathbb{R}^2$ and $v$ is the direction vector of the edge. All marked unbounded edges have direction $\vec{0}$ (the zero vector) and at each vertex a balancing condition is satisfied: the sum of the direction vectors of all flags at that vertex must be $\vec{0}$. 
\end{defn}

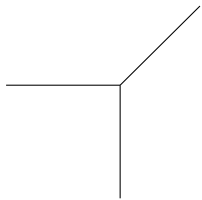
\begin{figure}[ht]
\centering
\begin{tikzpicture}[scale=1.5]

\draw (0, 0) -- (0.7, 0.7);

\draw (0, 0) -- (-1, 0);

\draw (0, 0) -- (0, -1);

\end{tikzpicture}
\setlength{\belowcaptionskip}{-3pt}
\caption{Planar tropical curve with one vertex and three unmarked unbounded edges.}
\end{figure}

\begin{defn}
By the definition of genus, the subgraph of a planar tropical curve $(\Gamma, h, x_1, x_2, \ldots , x_n)$ of genus $1$ contains exactly one cycle formed of vertices and bounded edges. It is even possible for this cycle to consist of one vertex with one bounded edge of direction $\vec{0}$, a self-loop (note that length of an edge can still be nonzero even if its direction vector is $\vec{0}$). The \textit{unique cycle} of a genus $1$ planar tropical curve $\Gamma$ is denoted by $\Gamma_c$ and referred to simply as ``the cycle''. 
\end{defn}

\begin{defn}
The \textit{combinatorial type} of a planar tropical curve consists of the combinatorial type of its abstract tropical curve and the directions of all of its edges. 
\end{defn}

\begin{defn}
The \textit{degree} of a planar tropical curve is the multiset of direction vectors of all unmarked unbounded edges. If the multiset of direction vectors consists of $(1, 1), (-1, 0), (0, -1)$ each occurring $d$ times, then we simply say the degree is $d$. For example, the curve in Figure 4 has degree $d = 1$. 
\end{defn}

\subsection{Moduli Space of Planar Curves}

\begin{defn}
The \textit{moduli space of planar tropical curves} $\mathcal{M}_{=1, n}(d)$ is defined as the geometric space whose points correspond to isomorphism classes of planar tropical curves of degree $d$ and genus $g = 1$ with $n$ marked edges. Note that since $d$ and $n$ are fixed, the number of unmarked unbounded edges and the number of marked unbounded edges are predetermined. However, the number of bounded edges and vertices may vary. 
\end{defn}

\begin{defn}
We may also define the moduli space of a particular combinatorial type $\alpha$. The \textit{moduli space of $\alpha$} is denoted by $\mathcal{M}_{=1, n}^{\alpha}(d)$, and is defined to be the subset of $\mathcal{M}_{=1, n}(d)$ consisting of all planar tropical curves of type $\alpha$. The topological structure of $\mathcal{M}_{=1, n}^{\alpha}(d)$ can be envisioned as an open convex polyhedron in a real vector space (see Lemma 3.1 in \cite{kerber2006counting}).
\end{defn}

\begin{rem}
Definition 3.2 in \cite{chan2017lectures} says that the moduli space $\mathcal{M}_{=1, n}(d)$ is the union of the $\mathcal{M}_{=1, n}^{\alpha}(d)$ for all different combinatorial types $\alpha$, with the boundaries of $\mathcal{M}_{=1, n}^{\alpha}(d)$ corresponding to planar tropical curves with fewer bounded edges.
\end{rem}

\begin{defn}
Two planar tropical curves $\Gamma_1, \Gamma_2 \in \mathcal{M}_{=1, n} (d)$ are \textit{connected} if there exists a path (Definition 3.28 in \cite{armstrong2013basic}) between the points corresponding to $\Gamma_1, \Gamma_2$ in the moduli space $\mathcal{M}_{=1, n} (d)$.
\end{defn}

\section{Preliminaries: Properties of Planar Tropical Curves}

We define a special map that allows us to deform a planar tropical curve into another curve with possibly fewer edges. We then use this map to simplify the planar tropical curves that we work with and narrow our scope to a certain subset of planar tropical curves. Next, we find an upper bound on the absolute values of the coordinates of the direction vectors of a genus $1$ planar tropical curve with degree $d$. Finally, we prove that the moduli space $\mathcal{M}_{=1, n}(1)$ is connected.

\subsection{Degeneration and Regeneration}

\begin{lemma}
\label{lemma3.1}
Let $C\in \mathcal{M}^{\alpha}_{=1,n}(d)$ be a planar tropical curve of degree $d$ with combinatorial type $\alpha$, and let $\Gamma_0^1=\{e_1,e_2,\ldots,e_k\}$ be all the bounded edges of $C$. We let $l_i=l(e_i)$ be the length of edge $e_i$. Assume we have a continuous family of tropical curves $f:(0,1)\to \mathcal{M}^{\alpha}_{=1,n}(d)$, and a subset $I\subset \{1,\ldots,n\}$ such that for $i\notin I$, $l_i(t)\to l_i'$ when $t\to 0$, and for $j\in I$, $l_j(t)\to 0$ when $t\to 0$. Here $l_i'$ are positive numbers for all $i\in I$. Then there exists a continuous map $\bar{f}:[0,1)\to \mathcal{M}_{=1,n}(d)$ s.t. $\bar{f}(t)=f(t)$ for $t\in(0,1)$ and $\bar{f}(0)$ has possibly fewer edges.
\end{lemma}

\begin{proof}
Recall the proof of Lemma 3.1 in \cite{kerber2006counting}, the moduli space $\mathcal{M}^{\alpha}_{=1,n}(d)$ is an open convex polyhedron in a real vector space. Thus, $f(t)$ has a limit for $l_i(t)=l_i'$ and $l_j(t)=0$ when $t\to 0$. We denote this limit by $\bar{f}(0)$. By Proposition 3.2 of \cite{kerber2006counting}, we know that $\bar{f}(0)$ represents a tropical curve with fewer edges.
\end{proof}

\begin{defn} 
Consider a planar tropical curve $\Gamma$ of degree $d$, genus $g$, and combinatorial type $\alpha$, along with another tropical curve $\Gamma'$  of genus $g$ and degree $d$ with fewer bounded edges than $\Gamma$. If there exists a continuous map $f:[0,1]\to \mathcal{M}_{=1,n}(d)$ such that $f(1)=\Gamma$, $f(t)$ has combinatorial type $\alpha$ for $0<t\leq 1$, and $f(0)=\Gamma'$, we say the continuous map $f$ is a $\textit{degeneration}$ of $\Gamma$. The inverse of $f$, the function $f^{1-t}$, will be called a \textit{regeneration} of $\Gamma'$.
\end{defn}

\begin{exmp}
Through the following degeneration of the graph $\Gamma$, we shrink edge $e$ while expanding the two edges that $e$ is connected to (the drawn arrow indicates the direction of shrinking), as seen in the middle figure. Finally, we obtain $\Gamma'$, the curve on the right:
\end{exmp}

\begin{figure}[ht]
\centering
\begin{tikzpicture}

\draw (1.6, 0) -- (2, 0);

\draw (2, 0) -- (2, 2);

\draw (2, 2) -- (0, 2);

\draw (0, 2) -- (0, 1.6);

\draw (0, 1.6) -- node[midway, left] {$e$} (1.6, 0);

\draw (2, 2) -- (2.5, 2.5);

\draw (2, 0) -- (2.5, 0.5);

\draw (1.6, 0) -- (1.6, -0.7);

\draw (0, 1.6) -- (-0.7, 1.6);

\draw (0, 2) -- (0.5, 2.5);

\end{tikzpicture}
\hspace{0.5cm}
\begin{tikzpicture}

\draw (1.6, 0) -- (2, 0);

\draw (2, 0) -- (2, 2);

\draw (2, 2) -- (0, 2);

\draw (0, 2) -- (0, 1.6);

\draw (1.6, 0) -- (0, 1.6);

\draw[dashed] (0, 0) -- (1.6, 0);

\draw[dashed] (0, 0) -- (0, 1.6);

\draw[dashed] (0, 1.2) -- (1.2, 0);

\draw[dashed] (0, 0.8) -- (0.8, 0);

\draw[dashed] (0, 0.4) -- (0.4, 0);

\draw[-latex] (1, 1) -- (0.1, 0.1);

\draw (2, 2) -- (2.5, 2.5);

\draw (2, 0) -- (2.5, 0.5);

\draw (1.6, 0) -- (1.6, -0.7);

\draw (0, 1.6) -- (-0.7, 1.6);

\draw (0, 2) -- (0.5, 2.5);

\draw[dashed] (1.2, 0) -- (1.2, -0.7);

\draw[dashed] (0, 1.2) -- (-0.7, 1.2);

\draw[dashed] (0.8, 0) -- (0.8, -0.7);

\draw[dashed] (0, 0.8) -- (-0.7, 0.8);

\draw[dashed] (0.4, 0) -- (0.4, -0.7);

\draw[dashed] (0, 0.4) -- (-0.7, 0.4);

\draw[dashed] (0, 0) -- (0, -0.7);

\draw[dashed] (0, 0) -- (-0.7, 0);

\end{tikzpicture}
\hspace{0.5cm}
\begin{tikzpicture}

\draw (0, 0) -- (2, 0);

\draw (2, 0) -- (2, 2);

\draw (2, 2) -- (0, 2);

\draw (0, 2) -- (0, 0);

\node at (0.2, 0.2) {$e$};

\draw (2, 2) -- (2.5, 2.5);

\draw (2, 0) -- (2.5, 0.5);

\draw (0, 2) -- (0.5, 2.5);

\draw (0, 0) -- (0, -0.7);

\draw (0, 0) -- (-0.7, 0);

\end{tikzpicture}
\setlength{\belowcaptionskip}{-3pt}
\caption{A degeneration of $\Gamma$.}
\end{figure}
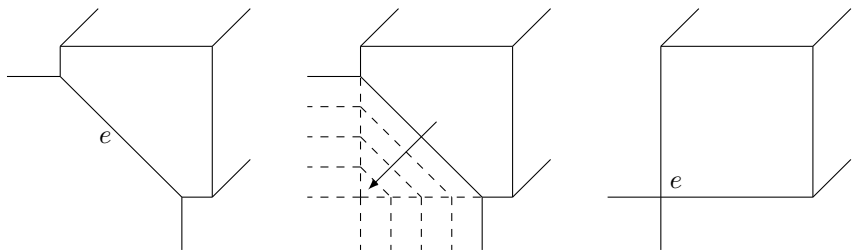

\begin{defn}
At the end of a degeneration $f$ from $\Gamma$ to $\Gamma'$, a bounded edge $e$ in $\Gamma$ whose length becomes $0$ is said to be \textit{shrunk}. A shrunken edge is mapped to a point in $\Gamma'$. The two vertices $V_1, V_2 \in \Gamma$ that are connected to $e$ in $\Gamma$ and are mapped to the same location in $\mathbb{R}^2$ by $f$ are said to be \textit{combined}. Suppose $f$ combines $V_1, V_2$ into a single vertex $V' \in \Gamma'$. In the regeneration $g = f^{1-t}$ that takes $\Gamma'$ to $\Gamma$, $V'$ is said to be $\textit{split}$ into $V_1$ and $V_2$, and edge $e$ is said to be a \textit{regenerated} edge.
\end{defn}



\begin{prop}
\label{prop3.1}
All planar tropical curves $\Gamma \in \mathcal{M}_{=1, n}(d)$ are connected to a planar tropical curve with no bounded edges other than those in $\Gamma_c$, the cycle of $\Gamma$.
\end{prop}

\begin{proof}
We can induct on the number of bounded edges that are not in $\Gamma_c$. If there are no bounded edges outside of $\Gamma_c$, we are done. Else, we may apply Lemma~\ref{lemma3.1} to shrink a bounded edge that is not in the cycle $\Gamma_c$ (this is possible because shrinking an edge not in $\Gamma_c$ does not change directions of other edges in $\Gamma$) and then use the induction hypothesis. Below is an illustration of shrinking a bounded edge not part of $\Gamma_c$: 

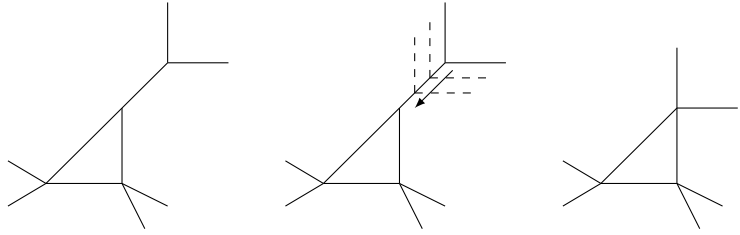
\begin{figure}[ht]
\centering
\begin{tikzpicture}

\draw (-1, 0) -- (0, 0);

\draw (-1, 0) -- (0, 1);

\draw (0, 0) -- (0, 1);

\draw (-1.5, 0.3) -- (-1, 0);

\draw (-1.5, -0.3) -- (-1, 0);

\draw (0, 0) -- (0.6, -0.3);

\draw (0, 0) -- (0.3, -0.6);

\draw (0, 1) -- (0.6, 1.6);

\draw (0.6, 1.6) -- (1.4, 1.6);

\draw (0.6, 1.6) -- (0.6, 2.4);

\end{tikzpicture}
\hspace{0.5cm}
\begin{tikzpicture}

\draw (-1, 0) -- (0, 0);

\draw (-1, 0) -- (0, 1);

\draw (0, 0) -- (0, 1);

\draw (-1.5, 0.3) -- (-1, 0);

\draw (-1.5, -0.3) -- (-1, 0);

\draw (0, 0) -- (0.6, -0.3);

\draw (0, 0) -- (0.3, -0.6);

\draw (0, 1) -- (0.6, 1.6);

\draw (0.6, 1.6) -- (1.4, 1.6);

\draw (0.6, 1.6) -- (0.6, 2.4);

\draw[dashed] (0.4, 1.4) -- (1.2, 1.4);

\draw[dashed] (0.4, 1.4) -- (0.4, 2.2);

\draw[dashed] (0.2, 1.2) -- (1, 1.2);

\draw[dashed] (0.2, 1.2) -- (0.2, 2);

\draw[-latex] (0.7, 1.5) -- (0.2, 1);

\end{tikzpicture}
\hspace{0.5cm}
\begin{tikzpicture}

\draw (-1, 0) -- (0, 0);

\draw (-1, 0) -- (0, 1);

\draw (0, 0) -- (0, 1);

\draw (-1.5, 0.3) -- (-1, 0);

\draw (-1.5, -0.3) -- (-1, 0);

\draw (0, 0) -- (0.6, -0.3);

\draw (0, 0) -- (0.3, -0.6);

\draw (0, 1) -- (0.8, 1);

\draw (0, 1) -- (0, 1.8);

\end{tikzpicture}
\setlength{\belowcaptionskip}{-3pt}
\caption{Shrinking a bounded edge not in $\Gamma_c$.}
\end{figure}
\end{proof}

\begin{rem}
We can use Proposition~\ref{prop3.1} to introduce a further simplification to our planar tropical curves by assuming that all bounded edges are part of $\Gamma_c$ (note that this implies all vertices are part of $\Gamma_c$ as well). By Proposition~\ref{prop3.1}, all planar tropical curves $\Gamma$ of genus $1$ and degree $d$ are connected to some planar tropical curve $\Gamma'$ with all bounded edges in $\Gamma'_c$, so to prove the connectedness of $\mathcal{M}_{=1, n} (d)$, we simply need to prove that the space of such curves $\Gamma'$ is connected. In the rest of the paper, we assume all planar tropical curves $\Gamma$ have all bounded edges part of the cycle $\Gamma_c$.
\end{rem}

\begin{defn}
For a vertex $V$ in a planar tropical curve $\Gamma$, let the \textit{UB-sum} of $V$ be the sum of the direction vectors of all unbounded edges attached to $V$.
\end{defn}

\begin{defn}
In a genus $1$ planar tropical curve $\Gamma$ with all bounded edges part of $\Gamma_c$, we define the vector $(x_i, y_i)$ as the UB-sum of each vertex $V_i$ in the cycle $\Gamma_c$. The \textit{effective degree} of this planar tropical curve is defined as $d' = \max(\max(|x_i|), \max(|y_i|))$ over all $x_i, y_i$, the largest absolute value of any of the $x$ or $y$ coordinates of the UB-sums of the vertices. Additionally, let the \textit{effective degree of the $x$ coordinates} be $d'_x = \max{(|x_i|)}$ over all $x_i$ and let the \textit{effective degree of the $y$ coordinates} be $d'_y = \max{(|y_i|)}$ over all $y_i$.
\end{defn}

\subsection{Shape of \texorpdfstring{$\Gamma_c$}{TEXT} and Marked Edges}

Recall that for any genus $1$ planar tropical curve $\Gamma$, the subgraph of $\Gamma$ has a unique cycle,  denoted by $\Gamma_c$.

\begin{prop}
\label{prop3.2}
Any genus $1$ planar tropical curve $\Gamma$ with $n$ marked unbounded edges and degree $d \ge 1$ is connected to another curve $\Gamma'$ such that all vertices $V$ in the curve $\Gamma'$ are connected to at least one unmarked, unbounded edge.
\end{prop}

\begin{proof}
Because $\Gamma$ has no bounded edges not in $\Gamma_c$, all vertices $V$ in $\Gamma$ are part of $\Gamma_c$, so each $V$ is connected to at most two bounded edges. Thus, it suffices to show that we may apply a series of degenerations to $\Gamma$ so that there is no vertex $V$ in $\Gamma$ that is connected to only bounded edges and marked unbounded edges. We induct on the number of such vertices. If there are no such vertices, then $\Gamma$ itself satisfies the desired property. We then take $\Gamma' = \Gamma$ and we're done. For the inductive step, we will show that given a vertex $V$ connected to only bounded edges and marked unbounded edges, we may shrink a bounded edge connected to $V$, reducing the number of vertices connected to only bounded edges and marked unbounded edges by one.

If a vertex $V \in \Gamma_c$ is connected to two bounded edges $e_1, e_2$ and only marked unbounded edges, then $e_1$ and $e_2$ must have the same direction by applying the balancing condition to vertex $V$. This implies that we may apply a degeneration process to shrink $e_1$ or $e_2$. This degeneration process combines vertex $V$ with another vertex, so the number of vertices connected to only bounded edges and marked unbounded edges decreases by one, completing our induction (the dotted edges below are marked, unbounded edges).

\begin{figure}[ht]
\centering
\begin{tikzpicture}[scale=0.7]

\draw (0, 0) -- (2, 1);

\draw (2, 1) -- (1, 3);

\draw (1, 3) -- (0, 0);

\draw (0, 0) -- (-1, 0);

\draw (0, 0) -- (0, -1);

\draw (2, 1) -- (2.7, 1.7);

\draw (1, 3) -- (0, 3);

\filldraw (1.5, 2) circle (0.03);

\draw[dashed] (1.5, 2) -- (2.5, 2.2);

\draw[dashed] (1.5, 2) -- (2.5, 2);

\node at (1.5, 2.6) {$e_1$};

\node at (2, 1.4) {$e_2$};

\end{tikzpicture}
\hspace{0.5cm}
\begin{tikzpicture}[scale=0.7]

\draw (0, 0) -- (2, 1);

\draw (2, 1) -- (1, 3);

\draw (1, 3) -- (0, 0);

\draw (0, 0) -- (-1, 0);

\draw (0, 0) -- (0, -1);

\draw (2, 1) -- (2.7, 1.7);

\draw (1, 3) -- (0, 3);

\filldraw (1.5, 2) circle (0.03);

\filldraw (1.33, 2.33) circle (0.03);

\filldraw (1.167, 2.667) circle (0.03);

\draw[dashed] (1.5, 2) -- (2.5, 2.2);

\draw[dashed] (1.5, 2) -- (2.5, 2);

\draw[dashed] (1.33, 2.33) -- (2.33, 2.53);

\draw[dashed] (1.33, 2.33) -- (2.33, 2.33);

\draw[dashed] (1.167, 2.667) -- (2.167, 2.867);

\draw[dashed] (1.167, 2.667) -- (2.167, 2.667);

\draw[-latex] (1.48, 1.8) -- (1, 2.8);

\node at (1.3, 3) {$e_1$};

\node at (1.8, 1.8) {$e_2$};

\end{tikzpicture}
\hspace{0.5cm}
\begin{tikzpicture}[scale=0.7]

\draw (0, 0) -- (2, 1);

\draw (2, 1) -- (1, 3);

\draw (1, 3) -- (0, 0);

\draw (0, 0) -- (-1, 0);

\draw (0, 0) -- (0, -1);

\draw (2, 1) -- (2.7, 1.7);

\draw (1, 3) -- (0, 3);

\filldraw (1, 3) circle (0.03);

\draw[dashed] (1, 3) -- (2, 3.2);

\draw[dashed] (1, 3) -- (2, 3);

\node at (1.7, 2.1) {$e_2$};

\end{tikzpicture}
\setlength{\belowcaptionskip}{-3pt}
\caption{Shrinking edge $e_1$.}
\end{figure}
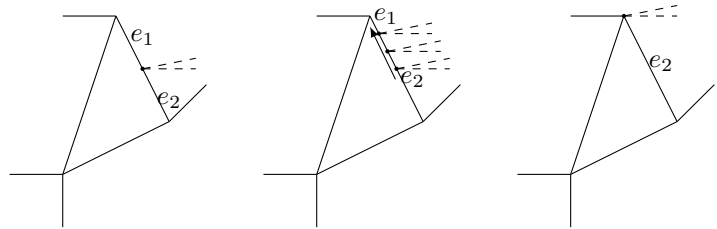

\end{proof}

\begin{prop}
\label{prop3.3}
Consider a degeneration or regeneration $f$ that maps a genus $1$ planar tropical curve $\Gamma$ to a different curve $\Gamma'$ such that each vertex in $\Gamma$ and $\Gamma'$ is connected to at least one unmarked unbounded edge. We may apply a similar degeneration/regeneration $f^*$ that maps a curve $\Gamma^*$ to $\Gamma^{*'}$, where $\Gamma^*$ is identical to $\Gamma$ but has no marked unbounded edges and similarly for $\Gamma^{*'}$ and $\Gamma'$.
\end{prop}

\begin{proof}
Recall that marked unbounded edges in planar tropical curves must have direction $(0, 0)$, so marked unbounded edges do not change the directions of any bounded or unmarked unbounded edges in $\Gamma_c$, as they do not change the balancing condition at each vertex. Thus, with any degeneration $f$, we can apply a similar degeneration $f^*$ that acts on a curve $\Gamma^*$ similar to $\Gamma$ but without marked unbounded edges. For a regeneration $f$ that splits a vertex $V$, we may distribute the marked unbounded edges at $V$ between the two new vertices formed however we want without changing the direction vectors of any other edges. Thus, the proposition is true for regenerations $f$ as well.
\end{proof}

\begin{rem}
When proving the connectedness of $\mathcal{M}_{=1, n}(d)$, we will prove that this space is path connected by applying a series of degenerations and regenerations to turn any curve $\Gamma \in \mathcal{M}_{=1, n}(d)$ into the planar tropical curve consisting of only one vertex $V$, a single bounded edge with direction $\vec{0}$ connected to $V$, and all other unbounded edges connected to $V$.

Proposition~\ref{prop3.2} and Proposition~\ref{prop3.3} allow us to prove the connectedness of $\mathcal{M}_{=1, n}(d)$ by proving that the space of of curves with no marked unbounded edges $\mathcal{M}_{=1, 0}(d)$ is connected, as Proposition~\ref{prop3.3} states that any series of degenerations and regenerations valid for a curve with no marked edges will apply to curves with marked edges.

In the rest of this paper, we assume that our planar tropical curves have no marked unbounded edges.
\end{rem}

\subsection{Bounding Direction Vectors of Bounded Edges}

\begin{lemma}
\label{lemma3.2}
In a genus $1$ planar tropical curve $(\Gamma, h, x_1, \ldots , x_n)$ of degree $d \ge 1$ with $k$ edges in $\Gamma_c$, the direction vector of each edge $e_i$ with $1 \le i \le k$ is $(a_i, b_i)$, with $|a_i|, |b_i| < d$.
\end{lemma}

\begin{proof}
Assume that the cycle $\Gamma_c$ contains $k$ vertices $V_1, V_2, \ldots , V_k$, with the UB-sum of $V_i$ being $\vec{x_i}$. Let the direction of the bounded edge connecting vertex $V_i$ and $V_{i + 1}$ and pointing towards vertex $V_{i + 1}$ be $\vec{v_i} = (a_i, b_i)$ and let its length be $l_i$ for $1 \le i \le k-1$. In addition, let $\vec{v_k} = (a_k, b_k)$ be the direction vector and $l_k$ be the length of the bounded edge starting at $V_k$ and pointing towards $V_1$. By the balancing condition at each $V_i$ and because the bounded edges in $\Gamma_c$ form a closed cycle, we have equations
$$\vec{v_2} = \vec{v_1} - \vec{x_2}$$
$$\vec{v_3} = \vec{v_2} - \vec{x_3}$$
$$\cdots$$
$$\vec{v_k} = \vec{v_{k-1}} - \vec{x_k}$$
$$l_1\vec{v_1} + l_2\vec{v_2} + \cdots + l_k\vec{v_k} = \vec{0}.$$
By substitution, we eliminate all the $\vec{v_i}$ except $\vec{v_1}$ and plug into the last equation to get
$$\vec{v_1} = \frac{(l_2 + \cdots + l_k)\vec{x_2} + (l_3 + \cdots + l_k)\vec{x_3} + \cdots + l_k\vec{x_k}}{l_1 + l_2 + \cdots + l_k}.$$
Because the degree of $\Gamma$ is $d$, we see that the absolute value of the coordinates of $\vec{v_1}$ can be no more than $\frac{(l_2 + \cdots + l_k)d}{l_1 + \cdots + l_k}$. Since the $l_i$ are positive reals, we conclude that this value is less than $d$, as desired.
\end{proof}

\subsection{A special case: \texorpdfstring{$d=1$}{TEXT}}

\begin{theorem}
\label{theorem3.3}
The moduli space $\mathcal{M}_{=1, n}(1)$ is connected for any nonnegative integer $n$.
\end{theorem}

\begin{proof}
By Lemma~\ref{lemma3.2}, we know that any bounded edge in a planar tropical curve of degree $1$ must have direction $(0, 0)$. By Lemma~\ref{lemma3.1}, we may apply a degeneration to shrink a bounded edge with direction $\vec{0}$ to a point. We shrink all but one of the bounded edges, as we need the genus to be $1$. The resulting planar tropical curve consists of one vertex with all unbounded edges attached to it as well as a bounded edge of direction $(0, 0)$ (the one that was not shrunken) and positive length. Because all degree $1$ planar tropical curves may be degenerated into this combinatorial type and because all of our degenerations were continuous maps, the moduli space $\mathcal{M}_{=1, n}(1)$ is connected.
\end{proof}

\section{Degeneration and Regeneration Methods}

We focus on three specific degeneration/regenerations that we will use to prove the connectedness of $\mathcal{M}_{=1, n}(d)$.

\subsection{Three Edge Degeneration}

\begin{exmp}
The following degeneration method involves changing the lengths of three consecutive bounded edges at once. To illustrate this special degeneration method, we shrink the horizontal edge  in the following curve:

\begin{figure}[ht]
\centering
\begin{tikzpicture}[scale=0.8]

\draw (0, -1) -- (-1, 0);

\draw (0, -1) -- (3, 0);

\draw (-1, 0) -- (0, 1);

\draw (3, 0) -- (2, 1);

\draw (0, 1) -- (2, 1);

\end{tikzpicture}
\hspace{0.5cm}
\begin{tikzpicture}[scale=0.8]

\draw (0, -1) -- (-1, 0);

\draw (0, -1) -- (3, 0);

\draw (-1, 0) -- (0, 1);

\draw (3, 0) -- (2, 1);

\draw (0, 1) -- (2, 1);

\draw[dashed] (0, 1) -- (1, 2);

\draw[dashed] (2, 1) -- (1, 2);

\draw[dashed] (0.33, 1.33) -- (1.67, 1.33);

\draw[dashed] (0.67, 1.67) -- (1.33, 1.67);

\draw[-latex] (1, 1.1) -- (1, 1.9);

\end{tikzpicture}
\hspace{0.5cm}
\begin{tikzpicture}[scale=0.8]

\draw (0, -1) -- (-1, 0);

\draw (0, -1) -- (3, 0);

\draw (-1, 0) -- (1, 2);

\draw (3, 0) -- (1, 2);

\end{tikzpicture}
\setlength{\belowcaptionskip}{-3pt}
\caption{Degeneration of horizontal edge.}
\end{figure}
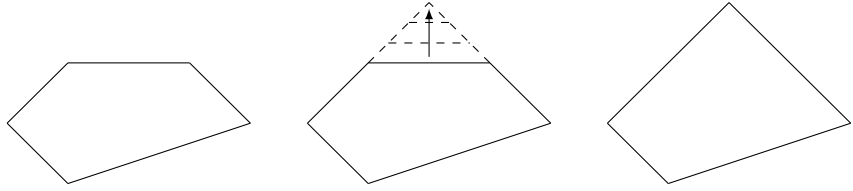
The goal is to use three-edge degeneration to shrink bounded edges in $\Gamma_c$ when $\Gamma_c$ contains more than three edges. Consider a set of any three consecutive edges, as shown below. We prove that we can use a degeneration to shrink at least one of these three consecutive bounded edges while keeping the two endpoints $1$ and $4$ (shown below) fixed. We keep endpoints fixed so that we do not change the direction vectors of the rest planar tropical curve, preserving the combinatorial type. 
\end{exmp}

\begin{figure}[ht]
\centering
\begin{tikzpicture}[scale=0.8]

\filldraw (-1, -1) circle (0.03);

\filldraw (3, -1) circle (0.03);

\draw (-1, -1) -- (0, 0);

\draw (0, 0) -- (2, 1);

\draw (2, 1) -- (3, -1);

\node at (-1.2, -1.2) {$1$};

\node at (-0.2, 0.2) {$2$};

\node at (2.1, 1.2) {$3$};

\node at (3.2, -1.2) {$4$};

\end{tikzpicture}
\setlength{\belowcaptionskip}{-3pt}
\caption{Three consecutive bounded edges.}
\end{figure}
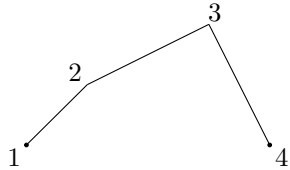

Label the vertices $1, 2, 3, 4$ from left to right and let $r_{ij}$ be the ray with endpoint at $i$ passing through $j$. let $r_1'$ be the ray with endpoint at $1$ pointing in the same direction as $r_{23}$ and let $r_4'$ be the ray with endpoint at $4$ pointing in the same direction as $r_{32}$.

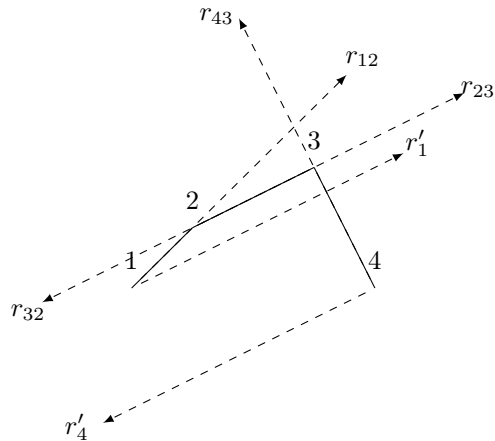
\begin{center}
\begin{tikzpicture}[scale=0.8]

\node (1) at (-1, -1) [label = above:1] {};

\node (2) at (0, 0) [label = above:2] {};

\node (3) at (2, 1) [label = above:3] {};

\node (4) at (3, -1) [label = above:4] {};

\node (5) at (1, -2) {};

\node (6) at (1, 0) {};

\draw (-1, -1) -- (0, 0);

\draw (0, 0) -- (2, 1);

\draw (2, 1) -- (3, -1);

\draw[dashed, -latex] (1) -- ($(1)!5cm!(2)$);

\draw[dashed, -latex] (4) -- ($(4)!5cm!(3)$);

\draw[dashed, -latex] (2) -- ($(2)!5cm!(3)$);

\draw[dashed, -latex] (3) -- ($(3)!5cm!(2)$);

\draw[dashed, -latex] (4) -- ($(4)!5cm!(5)$);

\draw[dashed, -latex] (1) -- ($(1)!5cm!(6)$);

\node at (2.8, 2.8) {$r_{12}$};

\node at (0.4, 3.5) {$r_{43}$};

\node at (4.7, 2.3) {$r_{23}$};

\node at (-2.7, -1.4) {$r_{32}$};

\node at (3.7, 1.4) {$r_1'$};

\node at (-1.9, -3.3) {$r_4'$};

\end{tikzpicture}
\end{center}

\begin{lemma}
\label{lemma4.1}
If one of the three intersections of rays $r_{12} \cap r_{43}$, $r_{12} \cap r_4'$, or $r_1' \cap r_{43}$ is non empty, then we can shrink one of the three bounded edges.
\end{lemma}

\begin{proof}

We do casework on the three types of intersections of the rays.

If $r_{12} \cap r_{43} \ne \emptyset$, we may shrink $e_{23}$ to zero, where $e_{ij}$ is the bounded edge connecting vertices $i$ and $j$:

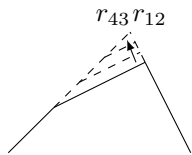
\begin{center}
\begin{tikzpicture}[scale=0.6]

\draw (-1, -1) -- (0, 0);

\draw[dashed] (0, 0) -- (1.67, 1.67);

\draw[dashed] (2, 1) -- (1.67, 1.67);

\draw[dashed] (0.56, 0.56) -- (1.89, 1.22);

\draw[dashed] (1.11, 1.11) -- (1.78, 1.44);

\draw[-latex] (1.8, 1) -- (1.6, 1.5);

\draw (0, 0) -- (2, 1);

\draw (2, 1) -- (3, -1);

\node at (2.1, 2) {$r_{12}$};

\node at (1.3, 2) {$r_{43}$};

\end{tikzpicture}
\end{center}

If $r_{12} \cap r_4' \ne \emptyset$, we may shrink $e_{34}$ to zero. Note: the two curves shown below are two different examples of the case where $r_{12} \cap r_4' \ne \emptyset$. Observe that in the second example $r_{12} \cap r_{43} \ne \emptyset$ as well, so we could have also used that type of degeneration. 

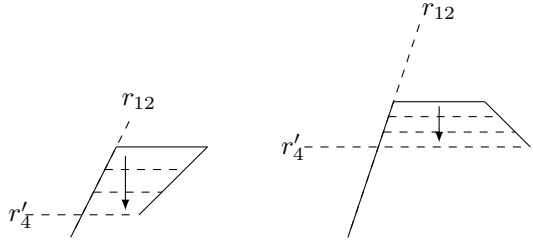
\begin{center}
\begin{tikzpicture}[scale=0.6]

\draw (0, 0) -- (1, 2);

\draw (1, 2) -- (3, 2);

\draw (3, 2) -- (1.5, 0.5);

\node (1) at (0, 0) {};

\node (2) at (1, 2) {};

\node (3) at (3, 2) {};

\node (4) at (1.5, 0.5) {};

\node (5) at (0, 0.5) {};

\draw[dashed] (1) -- ($(1)!3cm!(2)$);

\draw[dashed] (4) -- ($(4)!2.5cm!(5)$);

\draw[dashed] (0.75, 1.5) -- (2.5, 1.5);

\draw[dashed] (0.5, 1) -- (2, 1);

\draw[-latex] (1.2, 1.8) -- (1.2, 0.6);

\node at (1.5, 3) {$r_{12}$};

\node at (-1.1, 0.5) {$r_4'$};

\end{tikzpicture}
\hspace{0.5cm}
\begin{tikzpicture}[scale=0.6]

\draw (-1, -3) -- (0, 0);

\draw (0, 0) -- (2, 0);

\draw (2, 0) -- (3, -1);

\node (1) at (-1, -3) {};

\node (2) at (0, 0) {};

\node (3) at (2, 0) {};

\node (4) at (3, -1) {};

\node (5) at (0, -1) {};

\draw[dashed] (1) -- ($(1)!5cm!(2)$);

\draw[dashed] (4) -- ($(4)!5cm!(5)$);

\draw[dashed] (-0.22, -0.67) -- (2.67, -0.67);

\draw[dashed] (-0.11, -0.33) -- (2.33, -0.33);

\draw[-latex] (1, -0.1) -- (1, -0.9);

\node at (1, 2) {$r_{12}$};

\node at (-2.2, -1) {$r_4'$};

\end{tikzpicture}
\end{center}

If $r_{43} \cap r_1' \ne \emptyset$, we may shrink $e_{12}$ to zero. Note: again, we have two different examples to illustrate this case. In the first example, we see $r_{12} \cap r_{43} \ne \emptyset$, so we could have shrunk $e_{23}$ to zero, while in the second example, the only edge that may be degenerated is $e_{12}$. 

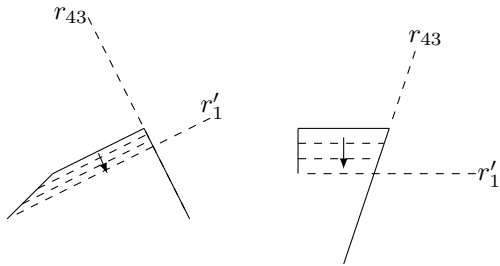
\begin{center}
\begin{tikzpicture}[scale=0.6]

\node (1) at (-1, -1) {};

\node (2) at (0, 0) {};

\node (3) at (2, 1) {};

\node (4) at (3, -1) {};

\node (5) at (1, -2) {};

\node (6) at (1, 0) {};

\draw (-1, -1) -- (0, 0);

\draw (0, 0) -- (2, 1);

\draw (2, 1) -- (3, -1);

\draw[dashed] (4) -- ($(4)!5cm!(3)$);

\draw[dashed] (1) -- ($(1)!5cm!(6)$);

\draw[dashed] (-0.33, -0.33) -- (2.067, 0.867);

\draw[dashed] (-0.67, -0.67) -- (2.133, 0.733);

\draw[-latex] (1, 0.45) -- (1.2, 0);

\node at (0.4, 3.5) {$r_{43}$};

\node at (3.5, 1.5) {$r_1'$};

\end{tikzpicture}
\hspace{0.5cm}
\begin{tikzpicture}[scale=0.6]

\node (1) at (0, 0) {};

\node (2) at (0, 1) {};

\node (3) at (2, 1) {};

\node (4) at (1, -2) {};

\node (5) at (1, 0) {};

\draw (0, 0) -- (0, 1);

\draw (0, 1) -- (2, 1);

\draw (2, 1) -- (1, -2);

\draw[dashed] (3) -- ($(4)!5cm!(3)$);

\draw[dashed] (1) -- ($(1)!4cm!(5)$);

\draw[dashed] (0, 0.33) -- (1.77, 0.33);

\draw[dashed] (0, 0.67) -- (1.88, 0.67);

\draw[-latex] (1, 0.8) -- (1, 0.1);

\node at (2.8, 3) {$r_{43}$};

\node at (4.2, 0) {$r_1'$};

\end{tikzpicture}
\end{center}
\end{proof}

\begin{defn}
In a planar tropical curve of degree $d > 1$, genus $1$, and at least four bounded edges in its cycle $\Gamma_c$, \textit{three-edge degeneration} is a degeneration that may be applied to three consecutive bounded edges in $\Gamma_c$ that shrinks the length of at least one of the three consecutive bounded edges to $0$.
\end{defn}

\begin{lemma}
\label{lemma4.2}

In a planar tropical curve with genus $1$ and $n > 3$ vertices in its cycle $\Gamma_c$ such that $\Gamma_c$ forms an (possibly self-intersecting) $n$-gon, we may apply a three-edge degeneration to any three consecutive bounded edges in the cycle, resulting in a polygon with fewer than $n$ sides.
\end{lemma}

\begin{proof}
If any one of the three bounded edges has direction $\vec{0}$, we may apply a degeneration to just that edge and shrink it to a point. Else, we claim that one of the three aforementioned intersections - $r_{12} \cap r_{43}, r_{12} \cap r_4', r_{43} \cap r_1'$ - always occurs. 

Without loss of generality, assume that $e_{23}$ is oriented horizontally. Furthermore, if we put our three consecutive edges in a coordinate plane with the x-axis along the same direction as $e_{23}$, we can assume that vertex $1$ has negative y-coordinate (if the y-coordinate was $0$, then the degeneration is trivial, as two edges point in the same direction). We now examine the possible locations of the fourth vertex and show that no matter where it is placed, one of the intersections of the rays is always nonempty. 

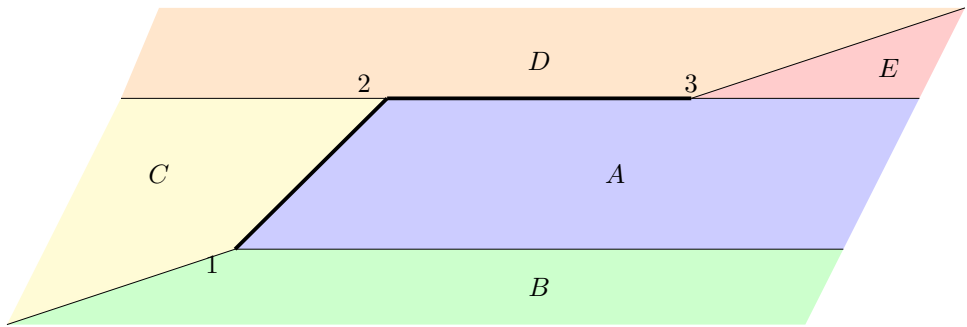
\begin{figure}[ht]
\centering
\begin{tikzpicture}

\fill[blue!20] (5, 0) -- (-2, 0) -- (-4, -2) -- (4, -2) -- cycle;

\fill[green!20] (4, -2) -- (-4, -2) -- (-7, -3) -- (3.5, -3) -- cycle;

\fill[yellow!20] (-5.5, 0) -- (-2, 0) -- (-4, -2) -- (-7, -3) -- cycle;

\fill[orange!20] (-5.5, 0) -- (-5, 1.2) -- (5.6, 1.2) -- (2, 0) -- cycle;

\fill[red!20] (5.6, 1.2) -- (2, 0) -- (5, 0) -- cycle;

\draw[line width=0.5mm] (-2, 0) -- (2, 0);

\draw[line width=0.5mm] (-4, -2) -- (-2, 0);

\draw[line width = 0.01mm] (-5.5, 0) -- (-2, 0);

\draw[line width = 0.01mm] (-4, -2) -- (-7, -3);

\draw[line width = 0.01mm] (2, 0) -- (5.6, 1.2);

\draw[line width = 0.01mm] (2, 0) -- (5, 0);

\draw[line width = 0.01mm] (-4, -2) -- (4, -2);

\node at (1, -1) {$A$};

\node at (0, -2.5) {$B$};

\node at (-5, -1) {$C$};

\node at (0, 0.5) {$D$};

\node at (4.6, 0.4) {$E$};

\node at (-4.3, -2.2) {$1$};

\node at (-2.3, 0.2) {$2$};

\node at (2, 0.2) {$3$};

\end{tikzpicture}
\setlength{\belowcaptionskip}{-3pt}
\caption{Casework on the location of the fourth vertex.}
\end{figure}

In the above picture, the bold lines are $e_{12}$ and $e_{23}$. We divide up the plane into regions with the following rays: $r_{23}$, $r_{32}$, $r_1'$, a ray at vertex $3$ with direction equal to the vector from vertex $1$ to vertex $3$, and a ray at vertex $1$ with direction equal to the vector from vertex $3$ to vertex $1$. The fourth vertex must lie in one of the five regions $A, B, C, D, E$. We claim that no matter where the fourth vertex is placed, we always get one of the three intersections $r_{12} \cap r_{43}, r_{12} \cap r_4',$ or $r_{43} \cap r_1'$. 

By inspection, if the vertex is placed in $A$, then ray $r_4'$ intersects $r_{12}$. If the vertex is placed in $B$, then ray $r_1'$ intersects $r_{43}$. If the vertex is placed in $C$, then the rays $r_{12}$ and $r_{43}$ intersect. If the vertex is placed in $D$, then the rays $r_1'$ and $r_{43}$ intersect. Finally, if the vertex is placed in $E$, then the rays $r_{12}$ and $r_{43}$ intersect. 

There are still several edge cases to be taken care of when the fourth vertex lies on any of the boundary lines. In these cases, we will still be able to shrink a bounded edge, and there is a possibility we may shrink more than one bounded edge. 

If the fourth vertex is placed on the bounded edge $e_{12}$ (as shown below), then we may shrink bounded edge $e_{23}$ and bounded edge $e_{34}$. 

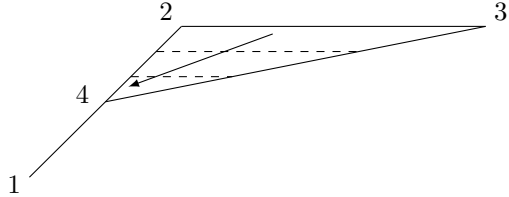
\begin{center}
\begin{tikzpicture}

\draw (-4, -2) -- (-2, 0);

\draw (-2, 0) -- (2, 0);

\draw (2, 0) -- (-3, -1);

\draw[dashed] (-7/3, -1/3) -- (1/3, -1/3);

\draw[dashed] (-8/3, -2/3) -- (-4/3, -2/3);

\draw[-latex] (-0.8, -0.1) -- (-2.7, -0.8);

\node at (-4.2, -2.1) {$1$};

\node at (-2.2, 0.2) {$2$};

\node at (2.2, 0.2) {$3$};

\node at (-3.3, -0.9) {$4$};

\end{tikzpicture}
\end{center}

If the vertex is placed on the bounded edge $e_{23}$, then we may shrink bounded edge $e_{34}$.

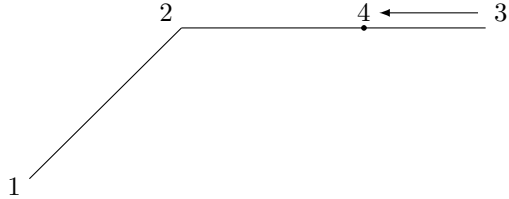
\begin{center}
\begin{tikzpicture}

\draw (-4, -2) -- (-2, 0);

\draw (-2, 0) -- (2, 0);

\draw[-latex] (1.9, 0.2) -- (0.6, 0.2);

\draw[fill] (0.4, 0) circle (0.3mm);

\node at (-4.2, -2.1) {$1$};

\node at (-2.2, 0.2) {$2$};

\node at (2.2, 0.2) {$3$};

\node at (0.4, 0.2) {$4$};

\end{tikzpicture}
\end{center}

If the vertex lies on the boundary of regions $A$ and $E$, we may still shrink bounded edge $e_{34}$. 

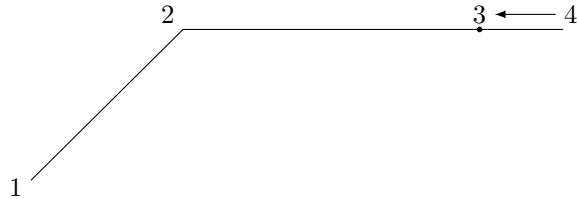
\begin{center}
\begin{tikzpicture}

\draw (-4, -2) -- (-2, 0);

\draw (-2, 0) -- (2, 0);

\draw (2, 0) -- (3, 0);

\draw[-latex] (2.9, 0.2) -- (2.1, 0.2);

\draw[fill] (1.9, 0) circle (0.3mm);

\node at (-4.2, -2.1) {$1$};

\node at (-2.2, 0.2) {$2$};

\node at (1.9, 0.2) {$3$};

\node at (3.1, 0.2) {$4$};

\end{tikzpicture}
\end{center}

If the vertex lies on the boundary of the regions $A$ and $B$, we may shrink bounded edges $e_{12}$ and $e_{34}$. 

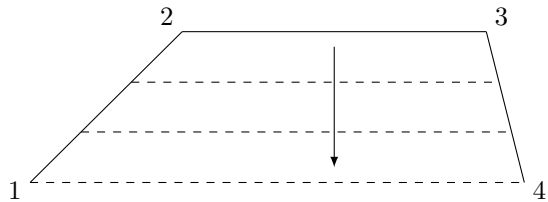
\begin{center}
\begin{tikzpicture}

\draw (-4, -2) -- (-2, 0);

\draw (-2, 0) -- (2, 0);

\draw (2, 0) -- (2.5, -2);

\draw[dashed] (-2.67, -0.67) -- (2.16, -0.67);

\draw[dashed] (-3.33, -1.33) -- (2.33, -1.33);

\draw[dashed] (-4, -2) -- (2.5, -2);

\draw[-latex] (0, -0.2) -- (0, -1.8);

\node at (-4.2, -2.1) {$1$};

\node at (-2.2, 0.2) {$2$};

\node at (2.2, 0.2) {$3$};

\node at (2.7, -2.1) {$4$};

\end{tikzpicture}
\end{center}

If the vertex lies on the boundary of the regions $B$ and $C$, we may shrink bounded edges $e_{12}$ and $e_{23}$. 

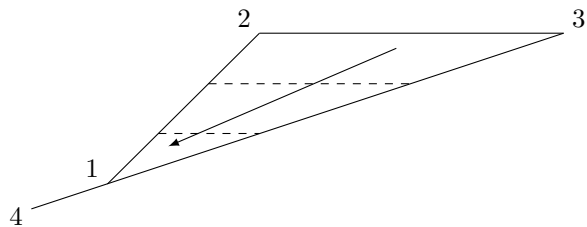
\begin{center}
\begin{tikzpicture}

\draw (-4, -2) -- (-2, 0);

\draw (-2, 0) -- (2, 0);

\draw (2, 0) -- (-5, -2.33);

\draw[dashed] (-2.67, -0.67) -- (0, -0.67);

\draw[dashed] (-3.33, -1.33) -- (-2, -1.33);

\draw[-latex] (-0.2, -0.2) -- (-3.2, -1.5);

\node at (-4.2, -1.8) {$1$};

\node at (-2.2, 0.2) {$2$};

\node at (2.2, 0.2) {$3$};

\node at (-5.2, -2.43) {$4$};

\end{tikzpicture}
\end{center}

If the vertex lies on the boundary of the regions $C$ and $D$, we may shrink bounded edge $e_{23}$. 

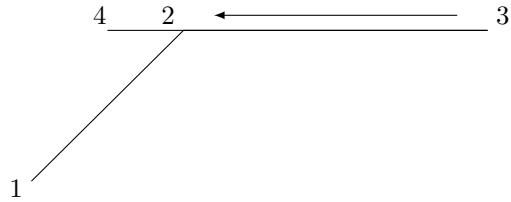
\begin{center}
\begin{tikzpicture}

\draw (-4, -2) -- (-2, 0);

\draw (-2, 0) -- (2, 0);

\draw (2, 0) -- (-3, 0);

\draw[-latex] (1.6, 0.2) -- (-1.6, 0.2);

\node at (-4.2, -2.1) {$1$};

\node at (-2.2, 0.2) {$2$};

\node at (2.2, 0.2) {$3$};

\node at (-3.1, 0.2) {$4$};

\end{tikzpicture}
\end{center}

If the vertex lies on the boundary between $D$ and $E$, we may shrink bounded edges $e_{12}$ and $e_{23}$.

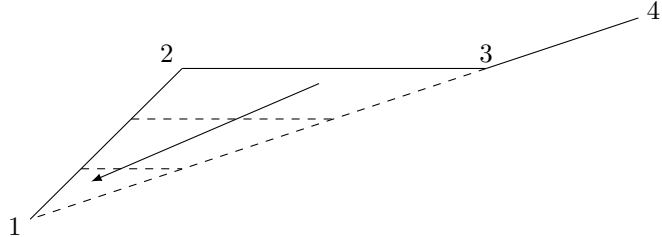
\begin{center}
\begin{tikzpicture}

\draw (-4, -2) -- (-2, 0);

\draw (-2, 0) -- (2, 0);

\draw (2, 0) -- (4, 0.67);

\draw[dashed] (2, 0) -- (-4, -2);

\draw[dashed] (-2.67, -0.67) -- (0, -0.67);

\draw[dashed] (-3.33, -1.33) -- (-2, -1.33);

\draw[-latex] (-0.2, -0.2) -- (-3.2, -1.5);

\node at (-4.2, -2.1) {$1$};

\node at (-2.2, 0.2) {$2$};

\node at (2, 0.2) {$3$};

\node at (4.2, 0.77) {$4$};

\end{tikzpicture}
\end{center}

\end{proof}

\subsection{Quadrilateral-Triangle Degeneration}
\begin{exmp}
Given a genus $1$ planar tropical curve $\Gamma$ with $4$ vertices $A, B, C, D$ where the cycle $\Gamma_c$ forms a non-self-intersecting quadrilateral, by extending two opposite sides of the quadrilateral and shrinking another bounded edge, we may turn the quadrilateral into a triangle. This degeneration is called \textit{quadrilateral-triangle degeneration}. In the following example, given convex quadrilateral $ABCD$, we shrink the edge $AB$, combining and identifying vertices $A$ and $B$.

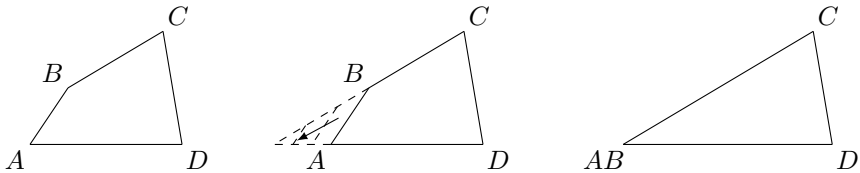
\begin{figure}[ht]
\centering
\begin{tikzpicture}

\draw (-1, 0) -- (1, 0);

\draw (1, 0) -- (3/4, 3/2);

\draw (3/4, 3/2) -- (-1/2, 3/4);

\draw (-1/2, 3/4) -- (-1, 0);

\node at (-1.2, -0.2) {$A$};

\node at (-0.7, 3/4 + 0.2) {$B$};

\node at (3/4 + 0.2, 1.7) {$C$};

\node at (1.2, -0.2) {$D$};

\end{tikzpicture} 
\hspace{0.5cm}
\begin{tikzpicture}

\draw (-1, 0) -- (1, 0);

\draw (1, 0) -- (3/4, 3/2);

\draw (3/4, 3/2) -- (-1/2, 3/4);

\draw (-1/2, 3/4) -- (-1, 0);

\draw[dashed] (-1, 0) -- (-1.75, 0);

\draw[dashed] (-0.5, 3/4) -- (-1.75, 0);

\draw[dashed] (-11/12, 0.5) -- (-5/4, 0);

\draw[dashed] (-4/3, 1/4) -- (-3/2, 0);

\draw[-latex] (-0.9, 0.35) -- (-1.45, 0.05);

\node at (-1.2, -0.2) {$A$};

\node at (-0.7, 3/4 + 0.2) {$B$};

\node at (3/4 + 0.2, 1.7) {$C$};

\node at (1.2, -0.2) {$D$};

\end{tikzpicture}
\hspace{0.5cm}
\begin{tikzpicture}

\draw (-7/4, 0) -- (1, 0);

\draw (1, 0) -- (3/4, 3/2);

\draw (3/4, 3/2) -- (-7/4, 0);

\node at (-2, -0.2) {$AB$};

\node at (3/4 + 0.2, 1.7) {$C$};

\node at (1.2, -0.2) {$D$};

\end{tikzpicture}
\setlength{\belowcaptionskip}{-3pt}
\caption{Degeneration of a convex quadrilateral.}
\end{figure}
\end{exmp}

\begin{exmp}
The same quadrilateral-triangle degeneration holds for a concave quadrilateral $ABCD$, as shown on the left. The degeneration shrinks the edge connecting vertices $C$ and $D$ and combines these two vertices. 

\begin{figure}[ht]
\centering
\begin{tikzpicture}[scale=0.8]

\draw (-2, 0) -- (0, 2);

\draw (0, 2) -- (0.5, 1);

\draw (0.5, 1) -- (2.5, 0);

\draw (2.5, 0) -- (-2, 0);

\node at (-2.2, -0.2) {$A$};

\node at (0, 2.2) {$B$};

\node at (0.7, 1.2) {$C$};

\node at (2.7, -0.2) {$D$};

\end{tikzpicture}
\hspace{0.5cm}
\begin{tikzpicture}[scale=0.8]

\draw (-2, 0) -- (0, 2);

\draw (0, 2) -- (0.5, 1);

\draw (0.5, 1) -- (2.5, 0);

\draw (2.5, 0) -- (-2, 0);

\draw[dashed] (0.5, 1) -- (1, 0);

\draw[dashed] (2/3, 2/3) -- (2, 0);

\draw[dashed] (5/6, 1/3) -- (3/2, 0);

\draw[-latex] (1.45, 0.45) -- (1.05, 0.05);

\node at (-2.2, -0.2) {$A$};

\node at (0, 2.2) {$B$};

\node at (0.7, 1.2) {$C$};

\node at (2.7, -0.2) {$D$};

\end{tikzpicture}
\hspace{0.5cm}
\begin{tikzpicture}[scale=0.8]

\draw (-2, 0) -- (0, 2);

\draw (0, 2) -- (1, 0);

\draw (1, 0) -- (-2, 0);

\node at (-2.2, -0.2) {$A$};

\node at (0, 2.2) {$B$};

\node at (1.2, -0.2) {$CD$};

\end{tikzpicture}
\setlength{\belowcaptionskip}{-3pt}
\caption{Degeneration of a concave quadrilateral.}
\end{figure}
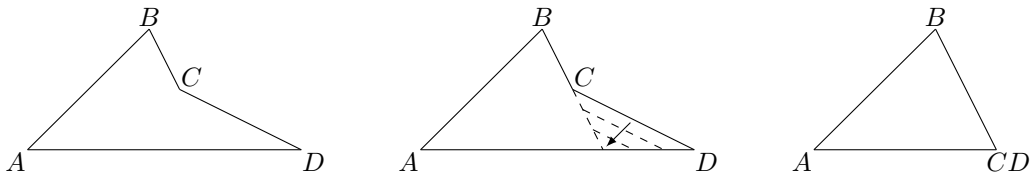

When applying this degeneration, we require that $ABCD$ be non self-intersecting so that we may combine adjacent vertices. In particular, two of the following pairs of vertices may always combine when $ABCD$ is non self-intersecting: $(A, B), (B, C), (C, D), (D, A)$. If we assume without loss of generality that $\angle A$ is the largest angle, then we know we can apply two different degenerations, one combining vertices $(A, B)$ and the other combining vertices $(A, D)$. However, if $ABCD$ is self intersecting, then we may not be able to degenerate the intersecting sides. For example, in the self-intersecting quadrilateral $ABCD$ shown below, we have no method of degenerating edge $BC$ or edge $AD$ without shrinking edge $AB$ or $CD$ first. However, this is not a concern, as we do not consider the degenerations of self-intersecting quadrilaterals in this paper.

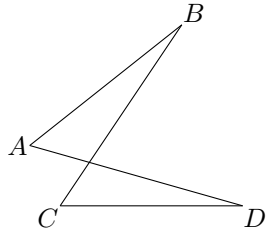
\begin{figure}[ht]
\centering
\begin{tikzpicture}[scale=0.8]

\draw (-0.5, 1) -- (2, 3);

\draw (2, 3) -- (0, 0);

\draw (0, 0) -- (3, 0);

\draw (3, 0) -- (-0.5, 1);

\node at (-0.7, 1) {$A$};

\node at (2.2, 3.2) {$B$};

\node at (-0.2, -0.2) {$C$};

\node at (3.2, -0.2) {$D$};

\end{tikzpicture}
\setlength{\belowcaptionskip}{-3pt}
\caption{Unable to degenerate a self-intersecting quadrilateral.}
\end{figure}
\end{exmp}

\subsection{Triangle-Quadrilateral Regeneration}

\begin{exmp} 
Consider the following planar tropical curve $\Gamma$ with three vertices such that $\Gamma_c$ forms a triangle. Suppose one of the vertices has two unmarked unbounded edges. We apply a regeneration to split this vertex into two new vertices, distributing the two unmarked unbounded edges between the two new vertices and turning the shape of the new curve $\Gamma'$ into a quadrilateral. Note that the choice of the distribution of the unmarked unbounded edges influences the direction of the regenerated bounded edge. This degeneration is called \textit{triangle-quadrilateral regeneration}.

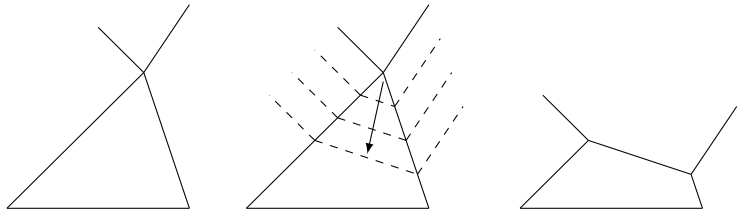
\begin{figure}[ht]
\centering
\begin{tikzpicture}[scale=0.6]

\draw (-2, 0) -- (2, 0);

\draw (-2, 0) -- (1, 3);

\draw (1, 3) -- (2, 0);

\draw (1, 3) -- (0, 4);

\draw (1, 3) -- (2, 4.5);

\end{tikzpicture}
\hspace{0.5cm}
\begin{tikzpicture}[scale=0.6]

\draw (-2, 0) -- (2, 0);

\draw (-2, 0) -- (1, 3);

\draw (1, 3) -- (2, 0);

\draw (1, 3) -- (0, 4);

\draw (1, 3) -- (2, 4.5);

\draw[dashed] (0.5, 2.5) -- (1.25, 2.25);

\draw[dashed] (0.5, 2.5) -- (-0.5, 3.5);

\draw[dashed] (1.25, 2.25) -- (2.25, 3.75);

\draw[dashed] (0, 2) -- (1.5, 1.5);

\draw[dashed] (0, 2) -- (-1, 3);

\draw[dashed] (1.5, 1.5) -- (2.5, 3);

\draw[dashed] (-0.5, 1.5) -- (1.75, 0.75);

\draw[dashed] (-0.5, 1.5) -- (-1.5, 2.5);

\draw[dashed] (1.75, 0.75) -- (2.75, 2.25);

\draw[-latex] (1, 2.8) -- (0.65, 1.2);

\end{tikzpicture}
\hspace{0.5cm}
\begin{tikzpicture}[scale=0.6]

\draw (-2, 0) -- (2, 0);

\draw (-2, 0) -- (-0.5, 1.5);

\draw (-0.5, 1.5) -- (1.75, 0.75);

\draw (1.75, 0.75) -- (2, 0);

\draw (-0.5, 1.5) -- (-1.5, 2.5);

\draw (1.75, 0.75) -- (2.75, 2.25);

\end{tikzpicture}
\setlength{\belowcaptionskip}{-3pt}
\caption{Regeneration of a triangle.}
\end{figure}
\end{exmp}

Not only do we wish to regenerate an edge and turn the shape of the graph from a triangle to a quadrilateral, we need to make sure this newly formed quadrilateral is not self-intersecting. Suppose that we have a planar tropical curve of genus $1$ with three vertices, and say we wish to split vertex $V$ into two. Let the two bounded edges that meet at $V$ have directions $\vec{a}, \vec{b}$, where the vectors point towards $V$. 

\begin{figure}[ht]
\centering
\begin{tikzpicture}[scale=0.8]

\draw [-latex] (-5, -1) -- node[midway, below] {$\vec{a}$} (0, 0);

\draw (0, 0) -- (5, 1);

\draw [-latex] (4, -2) -- node[midway, below] {$\vec{b}$} (0, 0);

\draw (4, -2) -- (-5, -1);

\draw (0, 0) -- (-4, 2);

\draw[-latex] (0, 0) -- node[at end, above] {$\vec{v}$} (-0.5, 2);

\node at (0.4, 1) {$A$};

\node at (3, -0.3) {$B$};

\node at (-0.2, -1) {$C$};

\node at (-3, 0.3) {$D$};

\end{tikzpicture}
\setlength{\belowcaptionskip}{-3pt}
\caption{Four regions formed by $\vec{a}$ and $\vec{b}$.}
\end{figure}
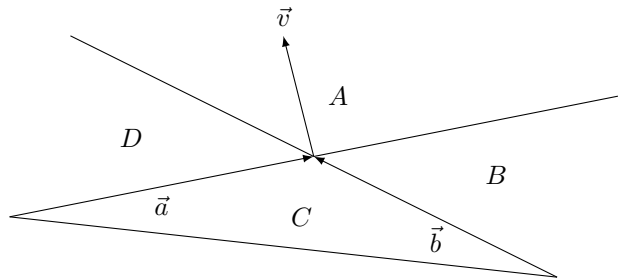

Suppose the direction of the new bounded edge we want to regenerate is $\vec{v}$ where $\vec{v}$ points towards the head of $\vec{b}$. If the newly split vertex at the head of the bounded edge with direction $\vec{a}$ has UB-sum $(c, d)$, then by the balancing condition at this  vertex, we see that $\vec{v}$ must have direction $\vec{a} - (c, d)$ (In most applications of this lemma, we make $(c, d)$ equal to $(1, 0)$, $(0, 1)$, or $(1, 1)$). Draw $\vec{v}$ with its tail at $V$. It is easy to see that if $\vec{v}$ lies in one of the regions $A, B, C$ then the quadrilateral we regenerate will NOT be self-intersecting, as we illustrate: \\

$\vec{v} \in A$: 

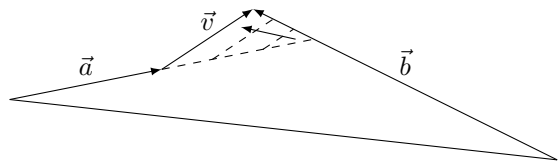
\begin{center}
\begin{tikzpicture}[scale=0.8]

\draw [-latex] (-5, -1) -- node[midway, above] {$\vec{a}$} (-2.5, -0.5); 

\draw [-latex] (-2.5, -0.5) -- node[midway, above] {$\vec{v}$} (-1, 0.5);

\draw [latex-] (-1, 0.5) -- node[midway, above] {$\vec{b}$} (4, -2);

\draw[dashed] (-2.5, -0.5) -- (0, 0);

\draw[dashed] (-5/3, -1/3) -- (-2/3, 1/3);

\draw[dashed] (-5/6, -1/6) -- (-1/3, 1/6);

\draw[-latex] (-0.3, 0) -- (-1.2, 0.2);

\draw (-5, -1) -- (4, -2);

\end{tikzpicture}
\end{center}

$\vec{v} \in B$: 

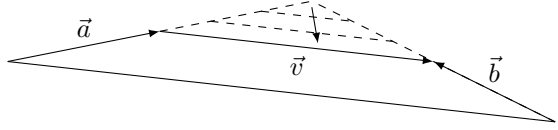
\begin{center}
\begin{tikzpicture}[scale=0.8]

\draw [-latex] (-5, -1) -- node[midway, above] {$\vec{a}$} (-2.5, -0.5); 

\draw [-latex] (-2.5, -0.5) -- node[midway, below] {$\vec{v}$} (2, -1);

\draw [latex-] (2, -1) -- node[midway, above] {$\vec{b}$} (4, -2);

\draw[dashed] (-2.5, -0.5) -- (0, 0);

\draw[dashed] (4, -2) -- (0, 0);

\draw[dashed] (-5/3, -1/3) -- (4/3, -2/3);

\draw[dashed] (-5/6, -1/6) -- (2/3, -1/3);

\draw[-latex] (0, -0.1) -- (0.1, -0.7);

\draw (-5, -1) -- (4, -2);

\end{tikzpicture}
\end{center}

$\vec{v} \in C$: 

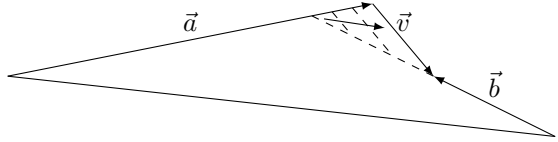
\begin{center}
\begin{tikzpicture}[scale=0.8]

\draw [-latex] (-5, -1) -- node[midway, above] {$\vec{a}$} (1, 0.2); 

\draw [-latex] (1, 0.2) -- node[midway, above] {$\vec{v}$} (2, -1);

\draw [latex-] (2, -1) -- node[midway, above] {$\vec{b}$} (4, -2);

\draw[dashed] (2, -1) -- (0, 0);

\draw[dashed] (1/3, 1/15) -- (2/3, -1/3);

\draw[dashed] (2/3, 2/15) -- (4/3, -2/3);

\draw[-latex] (0.2, -0.05) -- (1.2, -0.2);

\draw (-5, -1) -- (4, -2);

\end{tikzpicture}
\end{center}

$\vec{v} \in D$: 

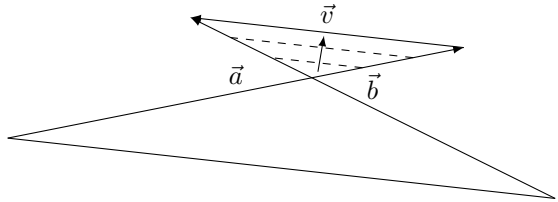
\begin{center}
\begin{tikzpicture}[scale=0.8]

\draw [-latex] (-5, -1) -- node[midway, above] {$\vec{a}$} (2.5, 0.5); 

\draw [-latex] (2.5, 0.5) -- node[midway, above] {$\vec{v}$} (-2, 1);

\draw [latex-] (-2, 1) -- node[midway, above] {$\vec{b}$} (4, -2);

\draw[dashed] (5/3, 1/3) -- (-4/3, 2/3);

\draw[dashed] (5/6, 1/6) -- (-2/3, 1/3);

\draw[-latex] (0.1, 0.1) -- (0.2, 0.7);

\draw (-5, -1) -- (4, -2);

\end{tikzpicture}
\end{center}

As we can see, when $\vec{v}$ lies in region $D$, on the boundary of regions $A, D$, or on the boundary of regions $C, D$, the quadrilateral we regenerate is self-intersecting. Similarly, if $\vec{v}$ starts at the head of $\vec{b}$ and points towards $\vec{a}$, then if $\vec{v}$ lies in region $B$, on the boundary of regions $A, B$, or on the boundary of regions $B, C$, the quadrilateral formed will be self-intersecting. 

\begin{defn}
If two vectors $\vec{a}$ and $\vec{b}$ have heads that meet at a point $V$, we define the \textit{bad region} of $\vec{a}$ with respect to $\vec{b}$ to be the region (along with its boundary) bounded by $\vec{a}$ and the extension of $\vec{b}$ past $V$, denoted by $\text{BR}(\vec{a}, \vec{b})$. Similarly, we define the bad region of $\vec{b}$ with respect to $\vec{a}$ to be the region bounded by $\vec{b}$ and the extension of $\vec{a}$ past $V$, denoted by $\text{BR}(\vec{b}, \vec{a})$. 
\end{defn}

\begin{rem}
If we refer to Figure 15, we see that $\text{BR}(\vec{a}, \vec{b})$ is region $D$ while $\text{BR}(\vec{b}, \vec{a})$ is region $B$. Generally, $\text{BR}(\vec{a}, \vec{b})$ is not the same as $\text{BR}(\vec{b}, \vec{a})$. Additionally, notice that $\text{BR}(\vec{a}, \vec{b})$ and the vector $\vec{b}$ are on the opposite sides of the line formed by $\vec{a}$. By this definition, if we want to regenerate a new bounded edge with direction $\vec{v}$ such that the tail of $\vec{v}$ is at the head of $\vec{a}$ and the head of $\vec{v}$ is at the head of $\vec{b}$, then this new quadrilateral is non-self-intersecting if and only if $\vec{v}$ does not lie in $\text{BR}(\vec{a}, \vec{b})$. 
\end{rem}

\begin{lemma}
\label{lemma4.3}
Given a planar tropical curve $\Gamma$ with genus $1$, degree $d \ge 2$, and three vertices, consider a vertex with at least two unbounded edges connected to it with nonzero UB-sum $(x, y)$. Then, we may always perform a triangle-quadrilateral regeneration to split the vertex and its unbounded edges to create two new vertices with UB-sum $(c, d)$ and $(e, f)$ with $(c + e, d + f) = (x, y)$ and obtain a non self-intersecting quadrilateral.
\end{lemma}

\begin{proof}
Suppose we wanted to split vertex $V$ so that one of the new vertices $V_1$ has UB-sum $(c, d)$, and $V_1$ is connected to the head of $\vec{a}$. Suppose the tail of the new bounded edge with direction vector $\vec{v}$ begins on $\vec{a}$ and points towards $\vec{b}$. Then we must have $\vec{a} - \vec{v} = (c, d)$, so $\vec{v} = \vec{a} - (c, d)$. Similarly, if the tail of $\vec{v}$ begins on $\vec{b}$ and points to $\vec{a}$, and the other new vertex $V_2$ has UB-sum $(c, d)$ instead, then we know $\vec{v} = \vec{b} - (c, d)$. This gives us two regeneration options. To prove the lemma we show that it is never the case that both regeneration options produce self-intersecting quadrilaterals. For now, assume that the tail of $\vec{v}$ starts on $\vec{a}$ and points to $\vec{b}$ so that $\vec{v} = \vec{a} - (c, d)$. 

For brevity, let ``ccw'' mean counterclockwise and let ``cw'' mean clockwise. If $\vec{a}$ lies to the left of the line determined by vector $(c, d)$ in the coordinate plane, then $\vec{a} - (c, d)$ lies ccw with respect to $\vec{a}$. If $\vec{a}$ lies to the right of the line determined by vector $(c, d)$ in the coordinate plane, then $\vec{a} - (c, d)$ lies cw with respect to $\vec{a}$:

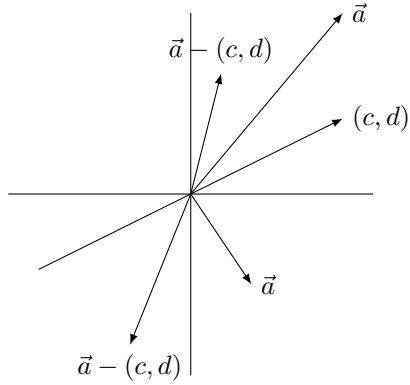
\begin{center}
\begin{tikzpicture}[scale=0.8]

\draw (-3, 0) -- (3, 0);

\draw (0, -3) -- (0, 3);

\draw[-latex] (-2.5, -1.25) -- (2.5, 1.25) node[at end, right] {$(c, d)$};

\draw[-latex] (0, 0) -- (2.5, 3) node[at end, right] {$\vec{a}$};

\draw[-latex] (0, 0) -- (0.5, 2) node[at end, above] {$\vec{a} - (c, d)$};

\draw[-latex] (0, 0) -- (1, -1.5) node[at end, right] {$\vec{a}$};

\draw[-latex] (0, 0) -- (-1, -2.5) node[at end, below] {$\vec{a} - (c, d)$};

\end{tikzpicture}
\end{center}

We let $\vec{a'}$ be a copy of $\vec{a}$ with its tail at the origin, for visualization purposes. Define $\vec{b'}$ in a similar way. Now assume without loss of generality that $\vec{a'}$ lies to the left of the line determined by vector $(c, d)$ the coordinate plane, so that $\vec{a'} - (c, d)$ lies in this region as well. As before, let the heads of $\vec{a}$ and $\vec{b}$ meet at a vertex $V$, which we take to be the origin. The bad region $\text{BR}(\vec{a}, \vec{b})$ is shown in red.

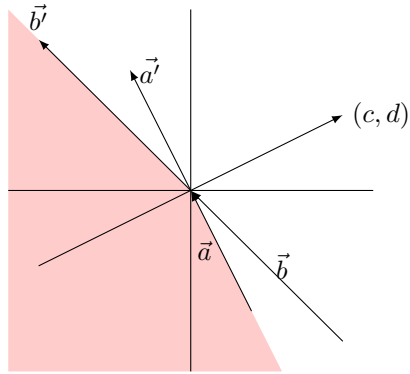
\begin{center}
\begin{tikzpicture}[scale=0.8]

\fill[red!20] (1.5, -3) -- (-3, -3) -- (-3, 3) -- (0, 0) -- cycle;

\draw (-3, 0) -- (3, 0);

\draw (0, -3) -- (0, 3);

\draw[-latex] (-2.5, -1.25) -- (2.5, 1.25) node[at end, right] {$(c, d)$};

\draw [-latex] (1, -2) -- node[midway, left] {$\vec{a}$} (0, 0);

\draw [-latex] (0, 0) -- node[at end, above] {$\vec{b'}$} (-2.5, 2.5);

\draw [latex-] (0, 0) -- node[midway, right] {$\vec{b}$} (2.5, -2.5);

\draw [-latex] (0, 0) -- node[at end, right] {$\vec{a'}$} (-1, 2);

\end{tikzpicture}
\end{center}

If $\vec{a} - (c, d)$ does not lie in $\text{BR}(\vec{a}, \vec{b})$, then by definition we can regenerate our triangle into a non-self-intersecting quadrilateral. In the other case, when $\vec{a} - (c, d)$ lies in $\text{BR}(\vec{a}, \vec{b})$, this implies (by the definition of bad region) that $\vec{b'}$ must lie in the  region bounded by $\vec{a'}$ and $\vec{a} - (c, d)$. Thus, $\vec{b'}$ must lie in the orange region and to the left of the line defined by vector $(c, d)$:

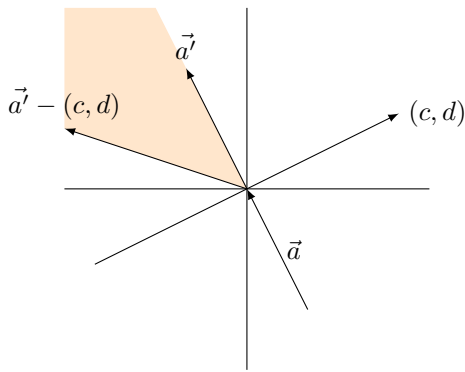
\begin{center}
\begin{tikzpicture}[scale=0.8]

\fill[orange!20] (-1.5, 3) -- (-3, 3) -- (-3, 1) -- (0, 0) -- cycle;

\draw (-3, 0) -- (3, 0);

\draw (0, -3) -- (0, 3);

\draw[-latex] (-2.5, -1.25) -- (2.5, 1.25) node[at end, right] {$(c, d)$};

\draw[-latex] (0, 0) -- (-1, 2) node[at end, above] {$\vec{a'}$};

\draw[-latex] (0, 0) -- (-3, 1) node[at end, above] {$\vec{a'} - (c, d)$};

\draw[latex-] (0, 0) -- (1, -2) node [midway, right] {$\vec{a}$};

\end{tikzpicture}
\end{center}

Assume that $\vec{b'}$ does indeed lie in this region. Now, since $\vec{b'}$ lies to the left of the line defined by the vector $(c, d)$, then $\vec{b'} - (c, d)$ lies ccw to $\vec{b'}$. However, since $\vec{b'}$ is ccw to $\vec{a'}$ we know both $\vec{a}$ and $\vec{b'} - (c, d)$ lie on the left side of $\vec{b}$ and its extension, so $\vec{b'} - (c, d)$ cannot lie in $\text{BR}(\vec{b}, \vec{a})$ (which is shown in red in the next image) by the second claim of Remark 4.1. This means the other regeneration method can still apply, the regeneration with the tail of $\vec{v}$ at the head of $\vec{b}$. A similar argument applies when $\vec{a}$ lies to the right of the line determined by vector $(c, d)$, simply replace every instance of counterclockwise with clockwise.

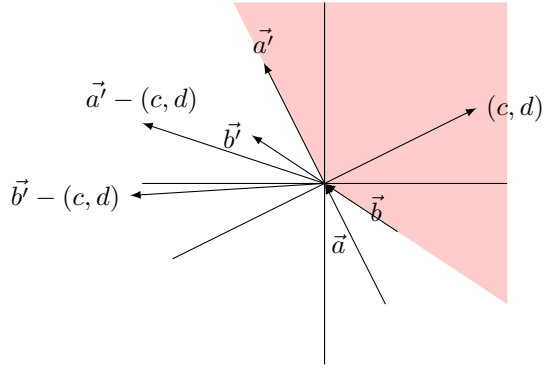
\begin{center}
\begin{tikzpicture}[scale=0.8]

\fill[red!20] (3, -2) -- (3, 3) -- (-1.5, 3) -- (0, 0) -- cycle;

\draw (-3, 0) -- (3, 0);

\draw (0, -3) -- (0, 3);

\draw[-latex] (-2.5, -1.25) -- (2.5, 1.25) node[at end, right] {$(c, d)$};

\draw[-latex] (0, 0) -- (-1, 2) node[at end, above] {$\vec{a'}$};

\draw[-latex] (0, 0) -- (-3, 1) node[at end, above] {$\vec{a'} - (c, d)$};

\draw[latex-] (0, 0) -- (1, -2) node [midway, left] {$\vec{a}$};

\draw[-latex] (0, 0) -- (-1.2, 0.8) node[at end, left] {$\vec{b'}$};

\draw[-latex] (1.2, -0.8) -- (0, 0) node[midway, right] {$\vec{b}$};

\draw[-latex] (0, 0) -- (-3.2, -0.2) node[at end, left] {$\vec{b'} - (c, d)$};

\end{tikzpicture}
\end{center}

In the above example, we can apply a regeneration that regenerates a new edge with direction vector $\vec{b} - (c, d)$ pointing from the head of $\vec{b}$ to the head of $\vec{a}$, changing the planar tropical curve from a triangle to a non self-intersecting quadrilateral:

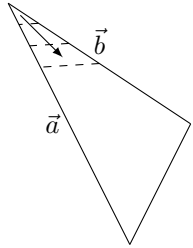
\begin{center}
\begin{tikzpicture}[scale=0.8]

\draw(0, 0) -- (3, -2) node [midway, above] {$\vec{b}$};

\draw (0, 0) -- (2, -4) node [midway, left] {$\vec{a}$};

\draw (3, -2) -- (2, -4);

\draw[dashed] (1/2, -1/3) -- (35/198, -35/99);

\draw[dashed] (1, -2/3) -- (35/99, -70/99);

\draw[dashed] (3/2, -1) -- (35/66, -35/33);

\draw[-latex] (0.2, -0.2) -- (0.9, -0.9);

\end{tikzpicture}
\end{center}

\end{proof}

\section{Connectedness of the Moduli Space of Planar Tropical Curves of Genus 1}

We prove two lemmas and use the regenerations and degenerations established in Section 4 to prove that the moduli space of planar tropical curves $\mathcal{M}_{=1, n}(d)$ is connected.

\begin{lemma} 
\label{lemma5.1}
Given a multiset of direction vectors $S = \{ \vec{v_1}, \vec{v_2}, \ldots , \vec{v_k}\}$ with $k \ge 2$, all of which have direction $(1, 1), (-1, 0),$ or $(0, -1)$ with sum of direction vectors in the set equal to $(m, n)$, then for any vector $(mk, nk)$ with $0 < k < 1$ and $mk, nk$ both integers, we can always pick a subset of $S$ such that the sum of the direction vectors in the subset is equal to $(mk, nk)$. 
\end{lemma}

\begin{proof}
We do some casework on the signs of $(m, n)$. In all of the cases, we rely on the fact that $k < 1 \implies |mk| < |m|, |nk| < |n|$. 

Case 1: $m, n \le 0$. In this case we know that we must have at least $m$ copies of $(-1, 0)$ and at least $n$ copies of $(0, -1)$ in order for the sum of the direction vectors to be $(m, n)$, so obviously we can choose a subset composed of $(-1, 0)$ and $(0, -1)$ vectors to form $(mk, nk)$. 

Case 2: $m > 0, n \le 0$. The sum of the direction vectors in the set $S$ sum to $(m, n)$ so we have at least $m$ copies of $(1, 1)$. This also means that we have at least $|n| + m$ copies of $(0, -1)$. However, in order to pick a subset of $S$ with direction sum $(mk, nk)$, we only need $mk < m$ copies of $(1, 1)$ and $|nk| + mk < |n| + m$ copies of $(0, -1)$, so we are able to pick a valid subset in this case. 

Case 3: $m \le 0, n > 0$. This is the same as case 2. 

Case 4: $m > 0, n > 0$. Without loss of generality we assume $m \ge n$. Then we must have at least $m$ copies of $(1, 1)$ in $S$ and at least $m - n$ copies of $(0, -1)$ in $S$. However, to pick a subset with direction sum equal to $(mk, nk)$, we only need $mk < m$ copies of $(1, 1)$ and $mk - nk < m - n$ copies of $(0, -1)$, so we can also pick a subset in this case. 
\end{proof}

\begin{lemma}
\label{lemma5.2}
Any genus $1$ planar tropical curve $\Gamma$ with degree $d > 1$ and $3$ vertices in its cycle is connected to a genus $1$ planar tropical curve with $3$ vertices and degree $d$ such that the effective degree of this curve is $d' \le 1$ or connected to a genus $1$ planar tropical curve with at most two vertices. 
\end{lemma}

\begin{proof}
Let $d'$ be the effective degree of our planar tropical curve $\Gamma$. We wish to apply a series of triangle-quadrilateral regenerations and quadrilateral-triangle degenerations to connect any planar tropical curves with degree $d$ and three vertices to a curve with three vertices and effective degree at most $1$ or a curve with at most two vertices. To do so, we induct on the effective degree of the graph. 

Base Case: $d' = 0, 1$. In this case, $d' \le 1$, so we are done.

Inductive Step: Without loss of generality assume that $d'$ appears in the $x$-coordinate of the UB-sum of some vertex $V$ so that $V$ has UB-sum equal to $(d', y)$. In addition, because $d'$ is the effective degree, we have $|y| \le |d'|$. We apply Lemma~\ref{lemma4.3} to split $V$ into two vertices connected by an edge, regenerating this new edge and turning our triangle into a non-self-intersecting quadrilateral. If the regenerated quadrilateral is a trapezoid, we may easily shrink the two bounded edges that are not parallel to create a planar tropical curve with two vertices, as desired. If the quadrilateral regenerated is not a trapezoid, then we turn the quadrilateral into a different triangle using quadrilateral-triangle degeneration. We show that this quadrilateral-triangle degeneration does not increase the $y$ effective degree $d'_y$ and strictly decreases the $x$ effective degree $d'_x$ of the planar tropical curve. We do some casework based on the sign of $d$. 

Case $1$: If $d' > 0$ and $y < d'$, then we split $V$ into two new vertices, one vertex with two unmarked unbounded edges with directions $(1, 1)$ and $(0, -1)$ and the other with UB-sum $(d' - 1, y)$. Notice that because $d' > 0$ and $y < d'$, we are always guaranteed to have at least one unmarked unbounded edge with direction $(1, 1)$ and one unmarked unbounded edge with direction $(0, -1)$ connected to $V$. Let the other two vertices in the quadrilateral have UB-sums equal to $(a, b), (c, d)$. We know that $a, c \le 0$, else this is a contradiction of the maximality of $d'$, as $a + c + d' = 0$. Now when we degenerate our quadrilateral into a triangle once more, either $(1, 0)$ combines with $(a, b)$ or $(d'-1, y)$ combines with $(c, d)$. If $(1, 0)$ combines with $(a, b)$ then the new vertex has UB-sum $(a+1, b)$. Because $a \le 0$ and $b$ is unchanged, we know $d'_x$ decreases while $d'_y$ stays constant. If $(d' - 1, y)$  combines with $(c, d)$, the new vertex has UB-sum $(d' - 1 + c, y + d)$. However, as $c \le 0$ and $y + d = -b$, we know $d'_x$ decreases while $d'_y$ will not increase. 

Case $2$: If $d' > 0$ and $y = d'$, then we split $V$ into two new vertices, one vertex with a single unmarked unbounded edge with direction $(1, 1)$ and another vertex with UB-sum $(d' -1, y - 1)$. Because $d' > 0$ we are always guaranteed at least one unmarked unbounded edge with direction $(1, 1)$ connected to $V$. If the other vertices have UB-sums $(a, b), (c, d)$ then we must have $a, b, c, d \le 0$ else there will be a contradiction of the maximality of $d'$. When we degenerate our quadrilateral into a triangle, either $(1, 1)$ combines with $(a, b)$ or $(d' - 1, y - 1)$ combines with $(c, d)$. If $(1, 1)$ combines with $(a, b)$ the new vertex will have UB-sum $(a + 1, b + 1)$, but since $a, b \le 0$, we know $d'_x$ decreases and $d'_y$ does not increase. Similarly, if $(d' - 1, y - 1)$ combines with $(c, d)$, we get a vertex with UB-sum $(d' + c - 1, y + d - 1)$. Since $c, d \le 0$ and $d' - 1 = y - 1 > 0$, we know $d'_x$ will decrease while $d'_y$ will not increase. 

Case $3$: If $d' < 0$, then we split $V$ into two new vertices, one vertex with a single unmarked unbounded edge with direction $(-1, 0)$ and the other with UB-sum $(d' + 1, y)$. Let the other two vertices in our quadrilateral have directions $(a, b), (c, d)$ with $a, c \ge 0$, else there would be a contradiction of the maximality of $d'$. If $(a, b)$ combines with $(-1, 0)$ then we get a vertex with UB-sum $(a - 1, b)$. However, $a \ge 0$, so $d'_x$ will decrease while $d'_y$ remains unchanged. If $(c, d)$ combines with $(d' + 1, y)$ then we obtain a vertex with UB-sum $(c + d' + 1, y + d)$. However, $c \ge 0$ while $d' + 1 < 0$ so the $d'_x$ will decrease. The $y$-coordinate is $y + d = -b$, so $d'_y$ does not increase. 

We have shown that applying a regeneration and a degeneration to our planar tropical curve $\Gamma$ causes $d'_x$ to decrease and $d'_y$ to not increase, giving us a new curve $\Gamma'$. By the same reasoning, we may then apply a regeneration and a degeneration process that causes $d'_y$ to decrease and $d'_x$ to not increase, giving us another curve $\Gamma''$. Now, our resulting curve $\Gamma''$ has a strictly smaller effective degree than the effective degree of $\Gamma$, so we may apply the induction hypothesis, completing the proof.
\end{proof}

\begin{theorem}
\label{theorem5.3}
The moduli space $\mathcal{M}_{=1, n}(d)$ of degree $d$ and genus $1$ planar tropical curves with $n$ marked unbounded edges is connected.
\end{theorem}

\begin{proof}
Assume $d \ge 2$, as we have proven the case $d = 1$ in Theorem~\ref{theorem3.3}. By Proposition~\ref{prop3.2} and Proposition~\ref{prop3.3}, we only need to show $\mathcal{M}_{=1, 0}(d)$ is connected. To prove that the moduli space $\mathcal{M}_{=1, 0}(d)$ is connected, we show that any planar tropical curve $\Gamma$ in this moduli space is connected to the trivial curve consisting of one vertex, a single bounded edge with direction $(0, 0)$ linking that vertex to itself, and all unbounded edges attached to that vertex. 

We first repeatedly apply Lemma~\ref{lemma4.2} and use three edge degeneration on any three consecutive bounded edges in the cycle. Each time three edge degeneration is performed, we shrink either $1$ or $2$ bounded edges, repeating this until we have $1$, $2$, or $3$ vertices left in $\Gamma$. 

Case $1$: If there is only $1$ vertex in the cycle $\Gamma_c$, then it must be the trivial curve with a single bounded edge of direction $(0, 0)$ and all unbounded edges attached to this vertex, and we are done. 

Case $2$: If there are two vertices, the cycle consists of two bounded edges each attached to the two vertices. Because our the shape of the cycle $\Gamma_c$ must be a line, the direction vectors of the two bounded edges are scalar multiples of each other. If the two bounded edges have direction $(0, 0)$, we may apply a degeneration to shrink one of the bounded edges to $0$ and arrive at the one vertex case.

Assume the two bounded edges do not have direction $(0, 0)$. To prove we may always degenerate this curve into a curve with only one vertex, we make use of Lemma~\ref{lemma5.1}. Because we know that the degree of the planar tropical curve is at least $2$, some vertex $V$ must have at least two unbounded edges connected to it. Suppose that the two bounded edges in have directions $(c_1m, c_1n)$ and $(c_2m, c_2n)$ where $m, n \in \mathbb{Z}$ and $c_1, c_2 \in \mathbb{Z}_{+}$. Then we may apply Lemma~\ref{lemma5.1} with the set $S$ as the set of direction vectors of the unmarked unbounded edges connected to $V$ and create a subset $U$ of unmarked unbounded edges with direction summing up to $(c_1m, c_1n)$. Then we perform a regeneration at $V$ to create two new vertices $V_1, V_2$ such that $V_1$ is connected to the bounded edge with direction $(c_1m, c_1n)$ and contains the unmarked unbounded edges in $U$ and $V_2$ is connected to the bounded edge with direction $(c_2m, c_2n)$ and contains the unmarked unbounded edges in $S - U$. We see that the regenerated edge must have direction $(0, 0)$. We may now shrink the other two bounded edges through a degeneration to arrive at the trivial curve. 

Case $3$: If the result of the degeneration methods turns the planar curve into a triangle, we apply Lemma~\ref{lemma5.2} to degenerate the planar curve into either a curve with two vertices or a curve with effective degree $d' \le 1$. If the curve has two vertices, we may apply the argument in Case 2. On the other hand, if we have a planar tropical curve with three vertices and $d' \le 1$, we may utilize Lemma~\ref{lemma3.2}. We have a genus $1$ curve with three vertices and because $d' \le 1$, we may apply the exact same argument given in Lemma~\ref{lemma3.2} to conclude there must be a bounded edge with direction $\vec{0}$. Now we may degenerate our triangle into the trivial curve with one vertex, a single bounded edge with direction $(0, 0)$ that makes up a self-loop, and all unbounded edges connected to the lone vertex.

We conclude that we can degenerate any planar tropical curve of degree $d$ and genus $1$ into the trivial curve, which proves connectedness of $\mathcal{M}_{=1, n}(d)$.
\end{proof}

\section{Conclusion}

To summarize, we introduced a set of maps called degenerations and regenerations that allowed us to shrink and regenerate bounded edges in tropical curves. We then focused on three specific maps (namely three-edge degeneration, quadrilateral-triangle degeneration, and triangle-quadrilateral regeneration) and applied these maps to reduce any curve $\Gamma \in \mathcal{M}_{=1, n}(d)$ to a planar tropical curve with at most three vertices. Afterwards, we employed Lemma~\ref{lemma5.1} and Lemma~\ref{lemma5.2} to prove that any planar tropical curve with at most three vertices is connected to the trivial planar tropical curve with only one vertex, proving the connectedness of the moduli space $\mathcal{M}_{=1, n}(d)$.

We are contining the work on our conjuecture that the moduli space is connected for planar tropical curves of higher genus. Some of the degenerations we use in this paper are useful in this case but further machinery needs to be developed to prove connectedness. In addition, it would be interesting to investigate a different moduli space $\mathcal{M}_{1, n}(d)$ (defined in \cite{kerber2006counting}) of planar tropical curves with $n$ marked edges, degree $d$, and genus $\le 1$ and prove that this moduli space is connected. We expect that the theorems proved in this paper in combination with the result that the moduli space is connected for abstract tropical curves will imply that this space is connected. Another query we find interesting is the connectedness of the moduli space of curves that are mapped to $\mathbb{R}^1$ or $\mathbb{R}^3$ instead of planar tropical curves mapped to $\mathbb{R}^2$.

\section{Appendix}

We extend the result proven in this paper to linear tropical curves, the one-dimensional analogue of planar tropical curves.

\subsection{Defining Linear Tropical Curves}

We provide several definitions regarding linear tropical curves. Most of these definitions are variants of those found in Section 2.

\begin{defn}
A linear tropical curve $\Gamma$ consists of an underlying abstract tropical curve along with a mapping $h$ taking the linear tropical curve to $\mathbb{R}^1$. Linear tropical curves are denoted by $(\Gamma, h, x_1, x_2, \ldots, x_n)$, sometimes referred to as $\Gamma$ for short, the same way we denote planar tropical curves. However, to avoid confusion, all curves we refer to in this section will be linear tropical curves.  Each edge in $\Gamma$ is given an integral one-dimensional direction vector, a vector of the form $(a)$ where $a \in \mathbb{Z}$. For simplicity, we treat the one-dimensional integral direction vector as an integer. For a flag $F$ we denote the direction vector of $F$ as $v(F)$. Edges of length $l$ are mapped to $a + v \cdot l$ where $a$ is a point in $\mathbb{R}^1$ and $v$ is the direction vector of the edge. Marked unbounded edges have direction $0$ and the balancing condition at each vertex is still satisfied: the sum of the direction vectors of all flags at a vertex must be $0$.
\end{defn}

\begin{defn}
The $\textit{unique cycle}$ and $\textit{combinatorial type}$ of a linear tropical curve are defined the same way as for a planar tropical curve, as given in Definitions 2.6 and 2.7. The unique cycle of a genus $1$ linear tropical curve $\Gamma$ will also be denoted by $\Gamma_c$.
\end{defn}

\begin{defn}
The $\textit{degree}$ $d$ of a linear tropical curve $\Gamma$ is the multiset of directions vectors of all unmarked unbounded edges. If the multiset of direction vectors consists of $\vec{-1}, \vec{1}$ each occurring $d$ times, then we simply say that the degree is $d$.
\end{defn}

\begin{defn}
The $\textit{moduli space of linear tropical curves}$ with genus $1$, $n$ marked unbounded edges, and degree $d$ is the geometric space whose points correspond to isomorphism classes of these linear tropical curves and is denoted by $\mathcal{M'}_{=1, n}(d)$.
\end{defn}

\begin{defn}
Given a linear tropical curve of type $\alpha$, the $\textit{moduli space of }$ $\alpha$ is denoted by $\mathcal{M'}^{\alpha}_{=1, n}(d)$ and is defined to be the subset of $\mathcal{M'}_{=1, n}(d)$ consisting of all linear tropical curves of type $\alpha$. 
\end{defn}

\begin{defn}
Two linear tropical curves $\Gamma_1, \Gamma_2 \in \mathcal{M'}_{=1, n}(d)$ are $\textit{connected}$ if there exists a path between the points corresponding to $\Gamma_1, \Gamma_2$ in the moduli space $\mathcal{M'}_{=1, n}(d)$.
\end{defn}

\begin{defn}
Given a linear tropical curve $\Gamma \in \mathcal{M'}_{=1, n}(d)$ and a vertex $V \in \Gamma$, the \textit{UB-sum} of $V$ is the sum of the directions of the unbounded edges connected to $V$.
\end{defn}

\subsection{Connectedness of \texorpdfstring{$\mathcal{M'}_{=1, n}(d)$}{TEXT}}

We will prove that the moduli space $\mathcal{M'}_{=1, n}(d)$ of linear tropical curves with genus $1$, $n$ marked edges, and degree $d$ is connected. Again, our strategy is to prove all curves are connected to the trivial curve with only one vertex. We extend our results for planar tropical curves; many of these results still hold for linear tropical curves.

It is easy to see that Lemma~\ref{lemma3.1} remains valid if we consider linear tropical curves instead of planar tropical curves. Degenerations and regenerations can be applied to shrink edges of some linear tropical curve $\Gamma$ the same way they can be applied to planar tropical curves.

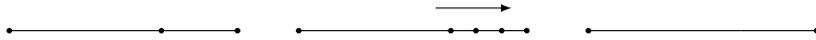
\begin{figure}[ht]
\centering
\begin{tikzpicture}

\draw (-1, 0) -- (1, 0);

\draw (1, 0) -- (2, 0);

\draw[fill=black] (-1, 0) circle (0.03 cm);

\draw[fill=black] (1, 0) circle (0.03 cm);

\draw[fill=black] (2, 0) circle (0.03 cm);

\end{tikzpicture}
\hspace{0.5cm}
\begin{tikzpicture}

\draw (-1, 0) -- (1, 0);

\draw (1, 0) -- (2, 0);

\draw[-latex] (0.8, 0.3) -- (1.8, 0.3);

\draw[fill=black] (-1, 0) circle (0.03 cm);

\draw[fill=black] (1, 0) circle (0.03 cm);

\draw[fill=black] (1.33, 0) circle (0.03 cm);

\draw[fill=black] (1.67, 0) circle (0.03 cm);

\draw[fill=black] (2, 0) circle (0.03 cm);

\end{tikzpicture}
\hspace{0.5cm}
\begin{tikzpicture}

\draw (-1, 0) -- (1, 0);

\draw (1, 0) -- (2, 0);

\draw[fill=black] (-1, 0) circle (0.03 cm);

\draw[fill=black] (2, 0) circle (0.03 cm);

\end{tikzpicture}
\setlength{\belowcaptionskip}{-3pt}
\caption{Degeneration of a linear tropical curve.}
\end{figure}

Degenerations and regenerations of linear tropical curves allow us to extend Proposition~\ref{prop3.1} to linear tropical curves as well. The proof of this proposition for linear tropical curves is the same as the proof for planar tropical curves. The modified Proposition~\ref{prop3.1} tells us that all linear tropical curves $\Gamma \in \mathcal{M'}_{=1, n}(d)$ are connected to a linear tropical curve with no bounded edges other than those in $\Gamma_c$.

We adjust Propositions~\ref{prop3.2} and~\ref{prop3.3} to hold for linear tropical curves. Again, the proofs are the same as the proofs for planar tropical curves.The modified Proposition~\ref{prop3.2} tells us that any genus $1$ linear tropical curve $\Gamma$ with $n$ marked unbounded edges and degree $d \ge 1$ is connected to another linear tropical curve $\Gamma'$ such that all vertices $V \in \Gamma'$ are connected to at least one unmarked unbounded edge. The modified Proposition~\ref{prop3.3} lets us  conclude that any series of degenerations and regenerations valid for a curve with no marked edges will apply to curves with marked edges and vice versa. Thus, when proving the connectedness of $\mathcal{M'}_{=1, n}(d)$, it suffices to show that the space of linear tropical curves with no marked edges and no bounded edges outside of $\Gamma_c$ is connected. We now assume that all linear tropical curves have no bounded edges oustide of $\Gamma_c$ and no marked edges.

We extend the idea of three-edge degenerations to linear tropical curves, renaming the process ``two-edge degeneration''. Notice that Figure 16 is an example of a two-edge degeneration.

\begin{defn}
In a genus $1$ linear tropical curve $\Gamma$ with at least three bounded edges in $\Gamma_c$, \textit{two-edge degeneration} is a degeneration that may be applied to any two consecutive bounded edges in $\Gamma_c$ that shrinks the length of one of the two edges to $0$.
\end{defn}

\begin{lemma}
\label{lemma7.1}
In a linear tropical curve $\Gamma$ with genus $1$ and $n > 2$ vertices in its cycle $\Gamma_c$, we may apply a two-edge degeneration to any two consecutive bounded edges in $\Gamma_c$.
\end{lemma}

\begin{proof}
The proof of this lemma is almost trivial in the one dimensional case. Suppose we have vertices $1, 2, 3$, and our two consecutive edges are $e_{12}$ and $e_{23}$ where $e_{ij}$ connects vertex $i$ and $j$. Either the bounded edges $e_{12}$ and $e_{23}$ intersect on their interiors or they don't. If the interiors of $e_{12}$ and $e_{23}$ do not intersect, we simply perform the degeneration illustrated in Figure 16. In the other case, we can still shrink one of the bounded edges.

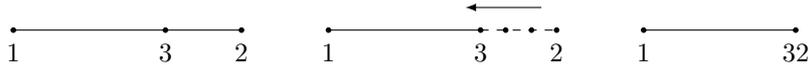
\begin{figure}[ht]
\centering
\begin{tikzpicture}

\draw (-1, 0) -- (1, 0);

\draw (1, 0) -- (2, 0);

\draw[fill=black] (-1, 0) circle (0.03 cm);

\draw[fill=black] (1, 0) circle (0.03 cm);

\draw[fill=black] (2, 0) circle (0.03 cm);

\node at (-1, -0.3) {$1$};

\node at (1, -0.3) {$3$};

\node at (2, -0.3) {$2$};

\end{tikzpicture}
\hspace{0.5cm}
\begin{tikzpicture}

\draw (-1, 0) -- (1, 0);

\draw[dashed] (1, 0) -- (2, 0);

\draw[latex-] (0.8, 0.3) -- (1.8, 0.3);

\draw[fill=black] (-1, 0) circle (0.03 cm);

\draw[fill=black] (1, 0) circle (0.03 cm);

\draw[fill=black] (1.33, 0) circle (0.03 cm);

\draw[fill=black] (1.67, 0) circle (0.03 cm);

\draw[fill=black] (2, 0) circle (0.03 cm);

\node at (-1, -0.3) {$1$};

\node at (1, -0.3) {$3$};

\node at (2, -0.3) {$2$};

\end{tikzpicture}
\hspace{0.5cm}
\begin{tikzpicture}

\draw (-1, 0) -- (1, 0);

\draw[fill=black] (-1, 0) circle (0.03 cm);

\draw[fill=black] (1, 0) circle (0.03 cm);

\node at (-1, -0.3) {$1$};

\node at (1, -0.3) {$32$};

\end{tikzpicture}
\setlength{\belowcaptionskip}{-3pt}
\caption{Two-edge degeneration.}
\end{figure}

\end{proof}

\begin{theorem}
The moduli space of genus $1$ linear tropical curves $\mathcal{M'}_{=1, n}(d)$ is connected for any nonnegative integer $n$ and any positive integer $d$.
\end{theorem}

\begin{proof}
We will use Lemma~\ref{lemma7.1} to show that any linear tropical curve $\Gamma \in \mathcal{M'}_{=1, n}(d)$ is connected to the trivial linear tropical curve with only one vertex $V$, one bounded edge with direction $0$, and all unbounded edges attached to $V$.

First we resolve the $d = 1$ case. In this case, we claim any $\Gamma \in \mathcal{M'}_{=1, n}(1)$ can only have one vertex. Because the curve only has two unmarked unbounded edges, it can have at most two vertices. We show that this is impossible. If there are two vertices $V_1, V_2$, without loss of generality let $V_2$ lie to the right of $V_1$. Then the unbounded edge with direction $1$ is connected to $V_2$ while the unbounded edge with direction $-1$ is connected to $V_1$. This implies that one of the bounded edges connecting $V_1$ and $V_2$ has direction $0$ while the other bounded edge has direction $1$ (with direction pointing towards $V_2$). This means the length of the edge with direction $1$ is $0$, contradiction.

If $d > 1$, we use two-edge degeneration to shrink all bounded edges in $\Gamma_c$ until there are only two vertices $V_1, V_2$ in $\Gamma_c$. Then, we regenerate a bounded edge with direction $0$ using a similar idea to the one in Lemma~\ref{lemma5.1}. Without loss of generality, assume $V_2$ lies to the right of $V_1$. Because $d > 1$, one of $V_1, V_2$ must be connected to at least $2$ unbounded edges. Without loss of generality, assume this is $V_2$ and let $V_2$ have UB-sum $x \in \mathbb{Z}_{\ge 0}$. Let the directions of the two bounded edges $e_1, e_2$ be $x_1, x_2 \in \mathbb{Z}_{\ge 0}$, respectively, where the directions point towards $V_2$ and $x_1 + x_2 = x$. There are at least $x$ unbounded edges with direction $1$ connected to $V_2$, so we may perform a regeneration by distributing $x_1$ unbounded edges with direction $1$ to a new vertex connected to $e_1$ and the rest of the unbounded edges to the other new vertex connected to $e_2$. Then, the newly regenerated edge must have direction $0$. We proceed by shrinking $e_1$ and $e_2$, arriving at the trivial curve. 

We conclude that we can degenerate any linear tropical curve of degree $d$ and genus $1$ into the trivial curve, proving the connectedness of $\mathcal{M'}_{=1, n}(d)$.
\end{proof}

\section{Acknowledgements}

First and foremost, I am grateful to the MIT PRIMES program for providing me with this opportunity to perform research. I would like to thank my mentor Yu Zhao for his guidance throughout the past year and for answering any questions that I had about my research topic or research in general. I would also like to thank Dr.~Khovanova for teaching me how to read and write research papers and work with LaTeX and Dr.~Makarova for her helpful feedback. 

\bibliographystyle{unsrt}
\bibliography{stuff}

\end{document}